\documentclass[11pt]{article}
\usepackage{amsthm, amsmath, amssymb, amsfonts, url, booktabs, tikz, setspace, fancyhdr, bm}

\usepackage{amsthm}

\usepackage{mathrsfs} 
\usepackage{xcolor}
\usepackage[margin = 2.3cm]{geometry}
\usepackage{hyperref, enumerate}
\usepackage[shortlabels]{enumitem}
\usepackage[babel]{microtype}
\usepackage[english]{babel}
\usepackage[capitalise]{cleveref}
\usepackage{comment}
\usepackage{bbm}
\usepackage{csquotes}
\usepackage{mathabx}
\usepackage{tikz}
\usepackage{graphicx}
\usepackage{float}
\usepackage{thm-restate}
\usepackage{mathtools}
\usetikzlibrary{arrows.meta}

\counterwithin{figure}{section}
\newcounter{shared}
\newtheorem{conjecture}[shared]{Conjecture}

\newtheorem{theorem}{Theorem}[section]
\newtheorem*{thm-non}{Theorem}
\newtheorem{prop}[theorem]{Proposition}
\newtheorem{lemma}[theorem]{Lemma}
\newtheorem{cor}[theorem]{Corollary}
\newtheorem{claim}[theorem]{Claim}

\theoremstyle{definition}
\newtheorem{defn}[theorem]{Definition}
\newtheorem*{defn-non}{Definition}

\newtheorem{ques}[shared]{Question}

\newlist{Case}{enumerate}{2}
\setlist[Case, 1]{%
    label           =   {\bfseries Case \arabic*.},
    labelindent=1em ,labelwidth=1.3cm, labelsep*=1em, leftmargin =!
}
\setlist[Case, 2]{%
    label           =   {\bfseries Subcase \arabic{Casei}.\arabic*.},
    labelindent=-1em ,labelwidth=1.3cm, labelsep*=1em, leftmargin =!
}

\newenvironment{poc}{\begin{proof}[Proof of claim]}{\end{proof}}

\newcommand{\C}[1]{{\protect\mathcal{#1}}}

\newcommand{\I}[1]{{\mathbbm #1}}

\newcommand{\ceil}[1]{\lceil #1\rceil}
\newcommand{\floor}[1]{\lfloor #1\rfloor}

\usepackage{todonotes} 

\newcommand{\eps}{\varepsilon}

\newcommand{\ex}{\mathrm{ex}}

\title{Clique density vs blowups}
\author{
Domagoj Brada\v{c}\thanks{Department of Mathematics, ETH, Z\"{u}rich, Switzerland. Research supported in part by SNSF grant 200021-228014. Email: domagoj.bradac@math.ethz}
\and
Hong Liu\thanks{Extremal Combinatorics and Probability Group (ECOPRO), Institute for Basic Science (IBS), Daejeon, South Korea. Emails: {\texttt \{hongliu, zixiangxu\}@ibs.re.kr}. Supported by IBS-R029-C4.}
\and
Zhuo Wu\thanks{Mathematics Institute and DIMAP, University of Warwick, Coventry, UK, and Extremal Combinatorics and Probability Group
(ECOPRO), Institute for Basic Science (IBS), Daejeon, South Korea. Supported by the Warwick Mathematics Institute Centre for Doctoral Training and funding from University of Warwick and the Institute for Basic Science (IBS-R029-C4). Email: zhuo.wu@warwick.ac.uk}
\and
Zixiang Xu\footnotemark[2]
}

\begin{document}
\maketitle

\begin{abstract}
    A well-known theorem of Nikiforov asserts that any graph with a positive $K_{r}$-density contains a logarithmic blowup of $K_r$. In this paper, we explore variants of Nikiforov's result in the following form. Given $r,t\in\mathbb{N}$, when a positive $K_{r}$-density implies the existence of a significantly larger (with almost linear size) blowup of $K_t$? Our results include:
\begin{itemize}
\item For an $n$-vertex ordered graph $G$ with no induced monotone path $P_{6}$, if its complement $\overline{G}$ has positive triangle density, then $\overline{G}$ contains a biclique of size $\Omega(\frac{n}{\log{n}})$. This strengthens a recent result of Pach and Tomon. 

For general $k$, let $g(k)$ be the minimum $r\in \mathbb{N}$ such that for any $n$-vertex ordered graph $G$ with no induced monotone $P_{2k}$, if $\overline{G}$ has positive $K_r$-density, then $\overline{G}$ contains a biclique of size $\Omega(\frac{n}{\log{n}})$. Using concentration of measure and the isodiametric inequality on high dimensional spheres, we provide constructions showing that, surprisingly, $g(k)$ grows quadratically. On the other hand, we relate the problem of upper bounding $g(k)$ to a certain Ramsey problem and determine $g(k)$ up to a factor of 2.


\item Any incomparability graph with positive $K_{r}$-density contains a blowup of $K_r$ of size $\Omega(\frac{n}{\log{n}}).$ This confirms a conjecture of Tomon in a stronger form. 
In doing so, we obtain a strong regularity type lemma for incomparability graphs with no large blowups of a clique, which is of independent interest.

We also prove that any $r$-comparability graph with positive $K_{(2h-2)^{r}+1}$-density contains a blowup of $K_h$ of size $\Omega(n)$, where the constant $(2h-2)^{r}+1$ is optimal.
\end{itemize} 
The $\frac{n}{\log n}$ size of the blowups in all our results are optimal up to a constant factor.
\end{abstract}

\section{Introduction}
\subsection{Overview}
The extremal number of a graph $H,$ denoted by $\ex(n, H)$, is the maximum number of edges in an $n$-vertex graph with no copy of $H$. Its historical roots trace back to Mantel's result from $1907$~\cite{1907Mantel}, showing that $\ex(n,K_{3}) = \lfloor \frac{n^2}{4} \rfloor.$ This was generalized by Tur\'{a}n in 1941~\cite{1941Turan} who determined $\ex(n, K_r)$ for all $r$. Later, the Erd\H{o}s--Stone--Simonovits theorem~\cite{1966ES, 1946ErodsBAMS} asserts that $\ex(n, H) = (1 - \frac{1}{\chi(H) - 1} + o(1)) {\binom{n}{2}}$, which asymptotically determines $\ex(n, H)$ for any graph $H$ that is not bipartite. The Erd\H{o}s--Stone--Simonovits theorem is widely considered as ``the 
fundamental theorem of extremal graph theory''. This was strengthened by Bollob\'{a}s, Erd\H{o}s, and Simonovits~\cite{1976JLMSBES} as follows. Denote by $\varrho_r(G)$ the \emph{$K_r$-density} of $G$, which is the number of copies of $K_r$ in $G$ normalized by ${\binom{|G|}{r}}$. For a graph $H$, we denote by $H[t]$ the \emph{$t$-blowup} of $H$ obtained by replacing every vertex of $H$ by an independent set of size $t$ and every edge of $H$ by a copy of $K_{t,t}$. It was shown~\cite{1976JLMSBES} that every $n$-vertex graph $G$ with edge density $\varrho_2(G)\ge 1-\frac{1}{r-1}+\eps$ contains a complete $r$-partite subgraph $K_{r}[\eps^{r}\log{n}]$. By considering random graphs of the same edge density, we see that the logarithmic dependence on $n$ is optimal. 

Given that random graphs are highly unstructured, it seems natural to ask whether larger blowups of cliques can be found in graphs from special families. We focus on finding largest possible blowups -- those of sizes almost linear in the order of $G$. Recently, there have been various results of this type under edge density conditions. For instance, Fox, Pach, and T\'{o}th~\cite{fox2010turan} showed that any $n$-vertex incomparability graph $G$ with $\varrho_2(G)>0$ contains a $K_2[t]$ with $t=\Omega(\frac{n}{\log{n}}).$ Tomon~\cite{2022DCGSharpString} proved that the complement of an $n$-vertex string graph $G$ with $\varrho_2(G)> \frac{3}{4}$ contains a $K_{2}[t]$ with $t=\Omega(n)$.
Apart from these recent developments, the Ramsey problem of finding a biclique of linear size in a graph or its complement also receives a great deal of attention~\cite{2023chostrong,fox2010turan,fu2023note}, as it is closely related to the Erd\H{o}s--Hajnal conjecture~\cite{1989DAMEHConj}.

In a different direction, a beautiful strengthening of the Bollob\'{a}s--Erd\H{o}s--Simonovits theorem by Nikiforov~\cite{2008BLMSNikiforov} replaces the edge density condition by a weaker clique density condition. More precisely, by the so-called supersaturation result of Erd\H{o}s and Simonovits~\cite{erdHos1983supersaturated}, any graph with $\varrho_2(G)> 1 - \frac{1}{r-1}$ not only contains a single $r$-clique, but has positive $K_r$-density $\varrho_r(G)> 0$. Nikiforov~\cite{2008BLMSNikiforov} showed that a graph having $\varrho_r(G)>0$ already suffices to imply a logarithmic blowup of $K_r$.


In this work we combine the above two directions of extensions to find large blowups of cliques in some well-known classes of graphs with clique density conditions instead of edge density. More precisely, we study the following general question.

\begin{ques}\label{ques:Main}
    Given a graph class $\mathcal{G}$ and $t\in\I N$, find the minimum $r\in \I N$ such that if any $G \in \mathcal{G}$ has positive $K_r$-density $\varrho_r(G)>0$,  then $G$ contains a near-linear-sized blowup of $K_t$?    
\end{ques}

\subsection{Our contributions}
\subsubsection{Ordered graphs forbidding monotone induced path}
An \emph{ordered graph} $G_{<}$ is a graph with a total ordering $<$ on its vertex set $V(G)$. We say $H_{<}$ is an \emph{induced ordered subgraph} of $G_{<}$ if there exists an order and adjacency preserving map from $V(H)$ to $V(G),$ i.e. there exists $f:V(H)\rightarrow V(G)$ such that for each pair $u,v\in V(H)$, if $u<v$ then $f(u)<f(v)$, and $uv\in E(H)$ if and only if $f(u)f(v)\in E(G)$. For extremal problems concerning ordered graphs, we refer the reader to some of the recent work~\cite{2019OrderedBDeg, 2017OrderedRamsey, 2019JCTAOrderedForest, 2022TuranOrderedGraph, 2006IJMOrdered} and the references therein.

Ordered graphs have close connections to geometrically defined graphs. 
Often, much better quantitative Ramsey results hold for graphs arising from geometric settings. For instance, 
Fox, Pach, and T\'oth~\cite{fox2010turan} proved that for any $n$-vertex intersection graph $G$ of $x$-monotone curves in the plane, either $G$ contains a $K_{2}[\frac{cn}{\log{n}}]$ or its complement $\overline{G}$ contains a $K_{2}[cn]$ for some $c>0$. The original proof in~\cite{fox2010turan} heavily relies on the geometric nature of the $x$-monotone curves.
Recently, Pach and Tomon~\cite{2019EuroCombTomonPach} provided a simpler combinatorial proof using ordered graphs. A \emph{monotone path} $P_{k}$ is an ordered graph with $k$ vertices $v_{1}<v_{2}<\cdots<v_{k}$ in which $v_{i}$ and $v_{j}$ are adjacent if and only if $|i-j|=1$. A key ingredient in their proof is the following result.

\begin{theorem}[Pach--Tomon~\cite{2019EuroCombTomonPach}]\label{thm:tomon-pach-ordered}
For every $k\ge 3$, there exists $0\le \beta(k)<1$ such that the following holds. Let $\eps>0$ and $G=G_{<}$ be an $n$-vertex ordered graph with no induced monotone path $P_{k}$ and its complement satisfies $\delta(\overline{G})\ge(\beta(k)+\eps)n$. Then $\overline{G}$ contains a copy of $K_{2}[\frac{cn}{\log{n}}]$, where $c=c(\eps)>0$.
\end{theorem}

This theorem generalizes the result of Fox, Pach, and T\'oth~\cite{fox2010turan} who proved it for the case $k=3$ (in this case $G$ is a comparability graph) with the optimal constant $\beta(3)=0$ (see~\cref{thm:FPT}). It remains an open problem to determine the optimal value of $\beta(k)$ for larger $k$. When $k=5$, Pach and Tomon~\cite{2019EuroCombTomonPach} provided the following example showing that $\beta(5)\ge 1/2$. Take the union of two disjoint cliques on $[n/2]$ and $[n]\setminus[n/2]$ respectively and add cross edges randomly and independently with probability $\eps>0$. The resulting graph $G$ has no induced monotone $P_5$ and with high probability $\delta(\overline{G})\ge (1/2 - \eps)n$, but the maximum size of a biclique in $\overline{G}$ is $O_{\eps}(\log n)$. 

We confirm that the above construction is optimal, i.e.,~$\beta(5)=1/2$. In fact, we prove it in a stronger form, showing that rather than the edge density condition as in~\cref{thm:tomon-pach-ordered}, large blowups originates from having positive triangle density.

\begin{theorem}\label{thm:p5}
    Let $\eps>0$ and $G=G_{<}$ be an $n$-vertex ordered graph with no induced monotone path $P_{5}$ and $\varrho_3(\overline{G})>\eps$. Then $\overline{G}$ contains a copy of $K_{2}[\frac{cn}{\log{n}}]$, where $c=c(\eps)>0$.
\end{theorem}

It is worth pointing out that this theme of \emph{``replacing edge density by clique density''} not only reveals the real cause of the phenonmenon occured (here being the appearance of large biclique), but also implies the stability result for the edge density version. Indeed, if an $n$-vertex ordered graph with no induced monotone $P_5$ has edge density close to $1/2$ and no biclique of size $\Omega(\frac{n}{\log n})$, then by~\cref{thm:p5} it must have zero triangle density. Thus, it follows from the classical Erd\H{o}s--Simonovits stability result~\cite{1966ES} that the graph $G$ must be close to balanced complete bipartite graph.

By considering the union of $s$ disjoint $\frac{n}{s}$-cliques and adding random edges between distinct cliques, we see that $\beta(2s+1)\ge 1-\frac{1}{s}$. After seeing~\cref{thm:p5}, it is tempting to believe that for an induced monotone $P_{2s+1}$-free $G$, positive $K_{s+1}$-density in $\overline{G}$ forces a large biclique in $\overline{G}$. Much to our own surprise, this is already false for the next case when forbidding $P_7$! It turns out that we need a positive $K_5$-density.

\begin{theorem}\label{thm:p7}
  For any $\eps>0$ and an $n$-vertex ordered graph $G=G_{<}$ with no induced monotone $P_{7}$, if $\varrho_5(\overline{G})>\eps$, then $\overline{G}$ contains a copy of $K_{2}[\frac{cn}{\log{n}}]$, where $c=c(\eps)>0$.

  On the other hand, there are $n$-vertex ordered graphs $G$ with no induced monotone path $P_{7}$, $\varrho_4(\overline{G})>0$ and no biclique of size $n/e^{o(\sqrt{\log n/\log\log n})}$ in $\overline{G}$.

\end{theorem}

\cref{thm:p5,thm:p7} are special cases of our main result. To state it, let $g(k)$ be the minimum $r\in \I N$ such that if any $n$-vertex ordered graph $G$ with no induced monotone $P_{2k}$ satisfies $\varrho_r(\overline{G})>0$, then $\overline{G}$ contains a copy of $K_2[\Omega(\frac{n}{\log n})]$. We in fact prove  $g(3)=3$ and $g(4)=5$, which imply~\cref{thm:p5,thm:p7} (with the weaker $P_6$-free or $P_8$-free condition). For general $k$, although the construction showing $\beta(2s+1)\ge 1-\frac{1}{s}$ is not optimal, it is not inconceivable that $g(k)$ is linear in $k$. However, we show that the growth rate of $g(k)$ is rather quadratic, determinging $g(k)$ up to a factor of 2.

Our main result reads as follows.

\begin{theorem} \label{thm:induced-intro}
   For $k \ge 3$, we have $\lfloor\frac{k^{2}}{4}\rfloor< g(k)\le \frac{k^2-k+2}{2}$. That is, the followings hold.
   \begin{itemize}
       \item For any $\eps > 0$ and an $n$-vertex ordered graph $G = G_<$ with no induced monotone $P_{2k}$, if $\varrho_{r}(\overline{G})>\eps$, where $r=\frac{k^2-k+2}{2}$, then $\overline{G}$ contains a copy of $K_{2}[\frac{cn}{\log{n}}]$, where $c=c(\eps)>0$.

       \item On the other hand, there are $n$-vertex ordered graphs $G$ with no induced monotone path $P_{2k-1}$, $\varrho_{\lfloor\frac{k^{2}}{4}\rfloor}(\overline{G})>0$ and no biclique of size $n/e^{o(\sqrt{\log n/\log\log n})}$ in $\overline{G}$.
   \end{itemize}
\end{theorem}

We remark that the $n/\log n$ blowup size in~\cref{thm:p5,thm:p7,thm:induced-intro} are optimal up to a constant factor by a construction in~\cite{fox2006bipartite}.

To prove~\cref{thm:induced-intro}, we reduce the problem of upper bounding $g(k)$ to a Ramsey problem (see~\cref{def:dependency-digraph}).
The lower bound construction is geometric and utilizes concentration of measure and isodiametric inequality on high dimensional spheres, motivated by the Bollob\'{a}s--Erd\H{o}s graph in Ramsey--Tur\'{a}n theory~\cite{2013IJMRT,1976BE,2015FoxLohZhao,2021LiuReiherRT}. It is interesting to see such a connection between these two seemingly unrelated problems.


\subsubsection{Graphs on posets}

Given a partially ordered set $(P,\prec)$, its \emph{(in)comparability graph} is a graph with vertex set $P$ in which two vertices form an edge if and only if the corresponding elements are (in)comparable in $P$. Fox, Pach and T\'{o}th~\cite{fox2010turan} proved the following result on existence of large complete bipartite graph in incomparability graphs with positive edge density.

\begin{theorem}[Fox--Pach--T\'{o}th~\cite{fox2010turan}]\label{thm:FPT}
    Let $\eps>0$ and $G$ be an $n$-vertex incomparability graph with $\varrho_2(G)>\eps$. Then $G$ contains a $K_{2}[\frac{cn}{\log{n}}]$, where $c=c(\eps)>0$.
\end{theorem}
Later this result was generalized for $k>2$ by Tomon~\cite{2016OrderTomon}, who showed that if an incomparability graph $G$ satisfies $\varrho_2(G)\ge 1-\frac{1}{9(k-1)}+\varepsilon$, then $G$ contains a copy of $K_{k}[\frac{cn}{(\log{n})^{s}}]$, where $c=c(k,\varepsilon)>0$ and $s=\lceil\log_{2}k\rceil$. Tomon proposed the following conjecture.
\begin{conjecture}[Tomon~\cite{2016OrderTomon}]\label{conj:Tomon'sConj}
   Let $\eps>0$ and $G$ be an $n$-vertex incomparability graph with $\varrho_2(G)>1-\frac{1}{k-1}+\eps$. Then $G$ contains a $K_{k}\big[\frac{cn}{(\log{n})^{s}}\big]$, where $c=c(\varepsilon,k)>0$  and $s=\lceil\log_{2}k\rceil$. 
\end{conjecture}

Our next result resolves~\cref{conj:Tomon'sConj} in a stronger form using the weaker clique density condition. The size of the blowup $n/\log n$ below is best possible up to a constant factor by~\cite{fox2006bipartite}.

\begin{restatable}{theorem}{thmincomparabilityrestate} \label{thm:main2}
Let $c \in (0, \frac{1}{2})$, $k\ge 2$ and $n$ be sufficiently large. If $G$ is an $n$-vertex incomparability graph with $\varrho_k(G)\ge c$, then $G$ contains a $K_{k}[\frac{c^4 n}{10^{8}k^{2}\log{n}}]$.
\end{restatable}

To prove Theorem~\ref{thm:main2}, we obtain a strong regularity type lemma for incomparability graphs not containing large blowups of $K_k$. We prove that such graphs admit regular partitions where almost all pairs of clusters are \emph{homogeneous}, that is, they form either complete or empty bipartite graphs. Such regulairty type result is of independent interest. A similar type of partition lemma was recently obtained for intersection graphs of pseudosegments by Fox, Pach, and Suk~\cite{2023FoxPachSuk}.

\begin{theorem} \label{lemma:incomparability-partition}
    Let $k \ge 2$ and $0<\varepsilon<1$. Suppose $n$ is sufficiently large and set $q=\frac{\varepsilon^4 n}{10^5 k^2 \log n}$ and $t = \frac{\varepsilon^7 n}{10^{11} k^5}.$ Let $G$ be an incomparability graph on $n$ vertices with no copy of $K_k[q].$ Then, there exists a partition $V(G) = V_0 \sqcup V_1 \sqcup \cdots \sqcup V_m$ such that $|V_0| \le \varepsilon n,$ $|V_1| = |V_2| = \dots = |V_m| = t$ and all but at most $\varepsilon m^2$ pairs $(V_i, V_j), 1 \le i < j \le m,$ are homogeneous.
\end{theorem}

Next, we consider unions of comparability graphs. Given partial orders $\prec_{1},\ldots,\prec_{r}$ on the same set $P$, the corresponding \emph{$r$-comparability graph} $G=G(\prec_{1},\ldots,\prec_{r})$ is the graph whose vertex set is $P$ and two elements $a,b\in P$ are adjacent if and only if $a\prec_{i} b$ or $b \prec_i a$ holds for some $1\le i\le r$. In other words, $G$ is the union of $r$ comparability graphs $G=\bigcup_{i\in[r]}G(\prec_{i})$.
\begin{restatable}{theorem}{thmmainrestate} \label{thm:main1}
Let $r,h\in \I N$ with $h\ge 2$, $k=(2h-2)^{r}+1$ and $\eps>0$. Let $G=\bigcup_{i\in[r]}G(\prec_{i})$ be an $n$-vertex $r$-comparability graph. If $\varrho_k(G)\ge\eps$, then some $G(\prec_{i})$, $i\in[r]$, contains a copy of $K_{h}[\frac{\varepsilon n}{r}]$.
\end{restatable}

This strengthens a result of Tomon~\cite{2016OrderTomon} which requires the stronger edge density condition $\varrho_2(G)\ge 1-\frac{1}{(2h-2)^r}+\eps$. The constant $k=(2h-2)^r+1$ in~\cref{thm:main1} is best possible by~\cite[Theorem~4(i)]{2016OrderTomon}. 

Both~\cref{thm:main2,thm:main1} also fall in the theme of ``replacing edge density by clique density''.

\subsubsection{Hereditary graphs and \texorpdfstring{$B_k$}{Bk} property}
Next we consider~\cref{ques:Main} for hereditary families. A family of graphs $\mathcal{G}$ is \emph{hereditary}, if for any graph $G\in \mathcal{G}$ and any induced subgraph $G'$ of $G$, $G'\in \mathcal{G}$. We introduce the following notion of the \emph{$B_{k}$ property}. A notion similar to the $B_{2}$ property has recently been introduced by Fox, Pach, and Suk~\cite{2023FoxPachSuk}. 

\begin{defn-non}
Let $k\ge 2$. A graph class $\mathcal{G}$ has the \emph {$B_k$ property} with function $f: \I R^+\rightarrow \I R^+$ if the following holds. For any $\eps>0$ and $n$-vertex $G\in \mathcal{G}$, if $\varrho_k(G)> \eps$, then $G$ contains a $K_{k}[f(\eps)n]$.
\end{defn-non}

Our next result states that the $B_2$ property implies the $B_k$ property for all $k \ge 2.$

\begin{theorem}\label{thm:B2ImpliesBk}
If a hereditary family of graphs has the $B_2$ property, then it also has the $B_k$ property for all $k \ge 2$.
\end{theorem}

Several classes of graphs are known to have the $B_{2}$ property. For instance, Fox, Pach, and T\'{o}th~\cite{fox2010turan} proved it for intersection graphs of planar convex sets as well as for incomparability graphs of posets with bounded dimension. Although the proof in~\cite{fox2010turan} is quite combinatorial, it is not clear how to adapt it to show e.g., that if the intersection graph of planar convex sets has positive triangle density, then it has a linear size complete tripartite graph, that is, such graphs have the $B_3$ property. Our~\cref{thm:B2ImpliesBk} circumvents ad-hoc family-specific arguments and show that all families with the $B_2$ property also have the $B_k$ property. 

Here we present another hereditary class with the $B_{2}$ property. The \emph{Vapnik-Chervonenkis dimension}, or \emph{VC-dimension} for short, of a graph and, more generally, a set system, is a fundamental parameter that measures its combinatorial  complexity. We defer its formal definition to~\cref{sec:Hereditary}.

\begin{theorem}\label{thm:VCDim1}
   The family of graphs with VC-dimension $d$ has the $B_2$ property if and only if $d=1$.
\end{theorem}



\noindent{\bf Structure of the paper. } In~\cref{sec:Pre}, we provide some useful tools and auxiliary results. The proof of~\cref{thm:induced-intro} is given in~\cref{section:OrderedMonotonePath,section:ConstructionBEGraph}. We prove~\cref{thm:main2,lemma:incomparability-partition,thm:main1} in~\cref{sec:Comparable},
and~\cref{thm:B2ImpliesBk,thm:VCDim1} in~\cref{sec:Hereditary}. Concluding remarks are given in~\cref{sec:conclusion}. 

\section{Tools and some auxiliary results}\label{sec:Pre}

All undirected graphs in this paper are simple. For the sake of clarity of presentation, we omit floor and ceiling signs whenever they are not essential.

Let $G$ be a graph. We will denote the set of vertices of $G$ by $V(G)$ and the set of edges by $E(G)$, and define $|G| := |V(G)|$ and $e(G):=|E(G)|$. For any vertex $v\in G$, we define $N(v)$ as the neighbor of $G$. We say a graph $G$ is \emph{$(\alpha,\beta)$-dense} if for every $S \subseteq V(G)$ with $|S| \ge \alpha |V(G)|$, $e(G[S]) \ge \beta \binom{|S|}{2}$ holds. 

For a graph $G$, let $X$, $Y$ be disjoint subsets of $V(G)$, and $e(X,Y)$ be the number of edges between $X$ and $Y$. We define the density of the pair $(X,Y)$ as $d(X,Y):=\frac{e(X,Y)}{|X|\cdot |Y|}$.
   A pair of vertex sets $X$ and $Y$ is said to be \emph{$\varepsilon$-regular}, if for all subsets $A\subseteq X$, $B\subseteq Y$ satisfying $|A|\geqslant \varepsilon|X|$, $|B|\geqslant \varepsilon|Y|$, we have
   \begin{equation*}
   \left|\frac{e(A,B)}{|A|\cdot|B|}-\frac{e(X,Y)}{|X|\cdot|Y|}\right|=\left|d(A,B)-d(X,Y)\right| \leqslant \varepsilon.
\end{equation*}
   A partition of $V$ into $k+1$ sets $(V_0,V_{1},\ldots,V_{k})$ is called an \emph{$\varepsilon$-regular partition}, if
   \begin{itemize}
     \item $|V_0|<\varepsilon n$ and for all $1\leqslant i<j\leqslant k$ we have $|V_{i}|=|V_{j}|$, and
     \item all except $\varepsilon k^{2}$ of the pairs $V_{i},V_{j}$, $1\leqslant i<j\leqslant k$, are $\varepsilon$-regular.
   \end{itemize}

The famous Szemer\'{e}di's regularity lemma~\cite{1978Regulariy} can be stated as follows.

\begin{lemma}[Szemer\'edi regularity lemma,~\cite{1978Regulariy}]\label{lem:psrl}
  For all $m,\varepsilon >0$, there exists an integer $M$ such that the following holds. If a graph $G$ has $n\geqslant M$ vertices, there exists an integer $K$ satisfying $m \leqslant K \leqslant M$, and  an \emph{$\varepsilon$-regular} partition $\mathcal{P}$ of $V(G)$ with $K+1$ parts.
\end{lemma}
We will also use the following simple lemma on the regular pairs.
\begin{lemma}[Slicing lemma,~\cite{komlos1995szemeredi}] \label{lem:slicing}
    Let $(A, B)$ be an $\varepsilon$-regular pair with density $d,$ and, for $\alpha > \varepsilon,$ let $A' \subseteq A$ and $B' \subseteq B$ be subsets with $|A'| \ge \alpha |A|, |B'| \ge \alpha |B|$. Then $(A', B')$ is an $\varepsilon'$-regular pair with $\varepsilon' = \max\{\varepsilon / \alpha, 2 \varepsilon\}$. Moreover, let $d'$ be the edge density of $(A',B')$, we have $|d' - d| < \varepsilon$.    
\end{lemma}


We will use multiple times the following simple consequence of the regularity lemma.

\begin{lemma} \label{lem:clique-in-reduced-graph}
    For every integer $k \ge 2,$ and any reals $\eta, \varepsilon > 0,$ there is a real $\gamma > 0$ such that the following holds. Let $G = G_<$ be an ordered graph on $n$ vertices with at least $\eta n^k$ copies of $K_k.$ Then, there are sets $A_1, \dots, A_k \subseteq V(G)$ with $A_1 < A_2 < \dots < A_k$ such that $|A_i| \ge \gamma n$ for all $i \in [k]$ and the pair $(A_i, A_j)$ is $\varepsilon$-regular in $G$ with density at least $\frac{\eta}{2}$ for all $1 \le i < j \le k.$
\end{lemma}
\begin{proof}
    Clearly we may assume that $\frac{1}{n} \ll \varepsilon \ll \eta\ll \frac{1}{k}$. Let $\varepsilon' = \frac{\varepsilon}{k}$ and $m = \frac{1}{\varepsilon'}$. We apply Lemma~\ref{lem:psrl} to $G$ with parameters $\varepsilon', m$ to obtain an $\varepsilon'$-regular partition $V_0, \dots, V_K.$ Note that there are at most $|V_0| n^{k-1} < \varepsilon n^k$ $k$-cliques in $G$ touching the set $V_0.$ Additionally, there are at most $\varepsilon K^2 |V_1|^2 \cdot n^{k-2} \le \varepsilon n^k$ $k$-cliques with one of the edges across an irregular pair. Finally, there are at most $K^2 |V_1|^2 \cdot \frac{2\eta}{3} \cdot n^{k-2} \le \frac{2\eta}{3} \cdot n^k$ $k$-cliques with one of the edges across pair with density at most $\frac{2\eta}{3}$. Hence, there exist $k$ sets, without loss of generality the sets $V_1, \dots, V_k$, such that each of the pairs $(V_i, V_j)$ is $\varepsilon$-regular with density at least $\frac{2\eta}{3}$.

    Denote $n_0 = |V_1| = \dots = |V_k|.$ We shall find the desired subsets $A_1, \dots, A_k$ with $A_{i}\subseteq V_{i}$ for $i=1,2\ldots,k$. For a vertex $x \in V(G),$ let $V_{\le x} = \{ y \in V(G), \, \vert \, y \le x\}$.

    For $i \in [k],$ set $U^1_i = V_i.$ Then, in steps $t = 1, \dots, k,$ we proceed as follows. Let $x$ be the minimum vertex in $V(G)$ such that for some $i, t \le i \le k,$ we have
    $|U_i^t \cap V_{\le x}| = \frac{n_{0}}{k}$. Without loss of generality, we may assume that $i = t.$ Then, set $A_i = U_i^t \cap V_{\le x}$ and for $j, i+1 \le j \le k,$ define $U^{t+1}_j = U^t_j \setminus V_{\le x}.$ Observe that by definition of $x$, we have $|U^{t+1}_j| \ge |U^t_j| - \frac{n_{0}}{k}$. It is easy to see that the described procedure produces sets $A_1, \dots, A_k$ with $A_1 < A_2 < \dots < A_k, |A_i| = \frac{n_{0}}{k}$ for all $i \in [k]$ and each of $A_i$ is a subset of a distinct set among $V_1, \dots, V_k.$ Finally, by Lemma~\ref{lem:slicing}, it follows that each of the pairs $(A_i, A_j)$ is $\varepsilon$-regular with density at least $\frac{2\eta}{3} - \varepsilon' > \frac{\eta}{2}$, as needed.
\end{proof}

Let $0<\gamma<1$ be a real number, we say a pair of vertex sets $(V_{1},V_{2})$ of $G$ is \emph{$\gamma$-homogeneous} if the density $d(V_1,V_2)$ is either less than $\gamma$ or larger than $1-\gamma$, where in the former case, we call $(V_1,V_2)$ \emph{$\gamma$-sparse}, and in the latter case, we call $(V_1,V_2)$ \emph{$\gamma$-dense}. We call the pair $(V_1, V_2)$ \emph{homogeneous} if $G[V_1, V_2]$ is empty or complete. A partition of $V(G)$ is called \emph{equitable} if every two parts differ in size by at most one.

We will use the following recent result of Fox and Pham~\cite{fox2024multipartite}.

\begin{lemma}[Fox--Pham~\cite{fox2024multipartite}]\label{lm:FHT}
Let $\ell \ge 2$ and $n\ge (100\ell)^5$. Every poset on $n$ elements contains $\ell$ disjoint sets $A_1,\ldots, A_\ell$ such that
\begin{itemize}
\item either $A_1\succ A_2\succ\cdots\succ A_\ell$ and $|A_i| = \frac{n}{10^{4}\ell^{5}}$ for $1\le i\le \ell,$ or
\item $A_i$ are pairwise incomparable and $|A_i| = \frac{n}{40\ell^{2}\log{n}}$ for $1\le i\le \ell$.
\end{itemize}
\end{lemma}

For a given $h$-dimensional unit sphere $\mathbb{S}^h$, we write $\lambda$ for the Lebesgue measure, which is normalized such that the unit sphere has Lebesgue measure $1$. For two subsets $A,B$ of a unit sphere, we define the Euclidean distance between them to be $d_{\max}(A,B):=\sup\{|\boldsymbol{a}-\boldsymbol{b}|:\boldsymbol{a}\in A,\boldsymbol{b}\in B\}$.
We also define the diameter of $A$ as $\textup{diam}(A):=d_{\max}(A,A)$.

We will take advantage of the following lemma in~\cite{2021LiuReiherRT}.
\begin{lemma}[\cite{2021LiuReiherRT}]\label{lemma:LargeMeasureToLargeDistance}
    Let $\mu\in (0,1)$ and $A,B\subseteq \mathbb{S}^{k-1}$ with $\lambda(A),\lambda(B)>e^{-\frac{k\mu}{2}}$, then $d_{\max}(A,B)\ge 2-\mu$. 
\end{lemma}

A \emph{spherical cap} is the smaller intersection of the unit sphere with a half-space. Given a spherical cap $C$ bounded by some hyperplane $H$, we call the point in $C$ with maximum Euclidean distance to $H$ the \emph{centre} of the spherical cap. The distance from the centre to $H$ is the \emph{height} of the spherical cap and the diameter of $C$ is the diameter of the intersection of $C$ and $H$. 

We will use the following lower and upper bounds on the measure of spherical caps.

\begin{lemma}[\cite{2021LiuReiherRT}]\label{lemma:Measure}
	For all $\delta>0$ and integers $k\ge 3$ , let $B\subseteq \mathbb{S}^{k-1}$ be the spherical cap consisting of all points with distance at most $\sqrt{2}-\delta/\sqrt{k}$ from a fixed point in $\mathbb{S}^{k-1}$. Then $\lambda(B)\ge 1/2-\sqrt{2}\delta$.
\end{lemma}

\begin{lemma}[\cite{2012AMMCaps}]\label{lem:cap-UB}
	Let $\alpha\in [0,1)$ and $C\subseteq \mathbb{S}^{k-1}$ be a spherical cap with height $1-\alpha$. Then $\lambda(C)\le e^{-k\alpha^2/2}$.
\end{lemma}

Recall that a spherical cap with height $1-\alpha$ has diameter $2\sqrt{1-\alpha^2}$.

The following folklore result partitions the sphere into small pieces of equal measure (see e.g.,~\cite{2002RSAPartition}).
\vskip -0.25em
\begin{lemma}\label{lem: partition-sphere}
	There exists $C>0$ such that the following holds. Let $0<\delta<1$ and $n\ge (C/\delta)^k$. Then $\mathbb{S}^{k-1}$ can be partitioned into $n$ pieces of equal measure, each of diameter at most $\delta$.
\end{lemma}

We also need the following geometric property, which plays a key role in the original Bollob\'{a}s--Erd\H{o}s graph~\cite{1976BE}.
\begin{theorem}\label{thm:BEIsK4free}
    For all $k\in\mathbb{N}$ and all $0<\mu<\frac{1}{4}$, there do not exist four points $p_{1},p_{2},q_{1},q_{2}\in\mathbb{S}^{k}$ such that $|p_{1}-p_{2}|\ge 2-\mu$, $|q_{1}-q_{2}|\ge 2-\mu$ and $|p_{i}-q_{j}|\le\sqrt{2}-\mu$ for all $i,j\in [2]$.
\end{theorem}

\section{Upper bound for Theorem~\ref{thm:induced-intro}}\label{section:OrderedMonotonePath}
In this section we prove the upper bound on $g(k)$ in~\cref{thm:induced-intro}. To achieve this, we first introduce a framework and a Ramsey variant based on some auxiliary directed graphs.

\subsection{A Ramsey problem}

\begin{defn} \label{def:dependency-digraph}
    Let $\chi$ be a red-blue edge-coloring of the ordered clique with vertices $v_1 < v_2 < \dots < v_k.$ Its \emph{dependency digraph} $D = D(\chi)$ is defined on the vertex set $\{v_1, \dots, v_k\}$ as follows. For every $i \in [k-1]$ and $j \in [k] \setminus \{i, i+1\},$ if $\chi(v_iv_{i+1}) = \mathrm{red}$ and $\chi(v_iv_j) = \chi(v_{i+1}v_j) = \mathrm{blue}$, then $D$ contains the directed edges $(v_j, v_i)$ and $(v_j, v_{i+1}).$ We say that $\chi$ is \emph{admissible} if $D(\chi)$ is acyclic.

    For $k \ge 1,$ let $f(k)$ denote the minimum integer $N$ such that in any red-blue edge-coloring of the ordered clique on $N$ vertices, there exist $k$ vertices such that the coloring induced on these $k$ vertices is admissible.
\end{defn}

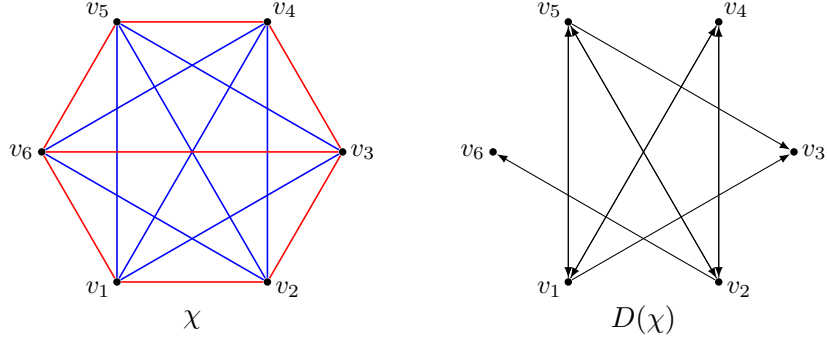
\begin{figure}[!ht]
    \centering
\begin{tikzpicture}[scale=1]
\node[inner sep= 1pt](z1) at (3,0)[circle,fill]{};
\node[inner sep= 1.3pt](v1) at (2.75,-0.1)[]{\small $v_1$};
\node[inner sep= 1pt](z2) at (5,0)[circle,fill]{};
\node[inner sep= 1.3pt](v2) at (5.27,-0.1)[]{\small $v_2$};
\node[inner sep= 1pt](z3) at (6, 1.73)[circle,fill]{};
\node[inner sep= 1.3pt](v3) at (6.27,1.73)[]{\small $v_3$};
\node[inner sep= 1pt](z4) at (5, 3.46)[circle,fill]{}; 
\node[inner sep= 1.3pt](v4) at (2.75,3.62)[]{\small $v_5$};
\node[inner sep= 1pt](z5) at (3, 3.46)[circle,fill]{};
\node[inner sep= 1.3pt](v5) at (5.23,3.62)[]{\small $v_4$};
\node[inner sep= 1pt](z6) at (2, 1.73)[circle,fill]{};
\node[inner sep= 1.3pt](v6) at (1.73,1.73)[]{\small $v_6$};
\node[inner sep= 1 pt]() at (4,-0.5)[]{ $\chi$}; 
\draw[red, line width=0.6pt] (z1) -- (z2);
\draw [red,line width=0.6pt](z2) -- (z3);
\draw [red,line width=0.6pt](z3) -- (z4);
\draw [red,line width=0.6pt](z4) -- (z5);
\draw [red,line width=0.6pt](z5) -- (z6);
\draw [red,line width=0.6pt](z6) -- (z1);
\draw[blue, line width=0.6pt] (z1) -- (z3);
\draw [blue,line width=0.6pt](z2) -- (z4);
\draw [blue,line width=0.6pt](z3) -- (z5);
\draw [blue,line width=0.6pt](z4) -- (z6);
\draw [blue,line width=0.6pt](z5) -- (z1);
\draw [blue,line width=0.6pt](z6) -- (z2);
\draw[blue, line width=0.6pt] (z1) -- (z4);
\draw [blue,line width=0.6pt](z2) -- (z5);
\draw [red,line width=0.6pt](z3) -- (z6);

\node[inner sep= 1pt](z1) at (9,0)[circle,fill]{};
\node[inner sep= 1.3pt](v1) at (8.75,-0.1)[]{\small $v_1$};
\node[inner sep= 1pt](z2) at (11,0)[circle,fill]{};
\node[inner sep= 1.3pt](v2) at (11.27,-0.1)[]{\small $v_2$};
\node[inner sep= 1pt](z3) at (12, 1.73)[circle,fill]{};
\node[inner sep= 1.3pt](v3) at (12.27,1.73)[]{\small $v_3$};
\node[inner sep= 1pt](z4) at (11, 3.46)[circle,fill]{}; 
\node[inner sep= 1.3pt](v4) at (8.75,3.62)[]{\small $v_5$};
\node[inner sep= 1pt](z5) at (9, 3.46)[circle,fill]{};
\node[inner sep= 1.3pt](v5) at (11.23,3.62)[]{\small $v_4$};
\node[inner sep= 1pt](z6) at (8, 1.73)[circle,fill]{};
\node[inner sep= 1.3pt](v6) at (7.73,1.73)[]{\small $v_6$};
\node[inner sep= 1 pt]() at (10,-0.5)[]{ $D(\chi)$}; 
\draw[-latex](z5) -- (z3);
\draw[-latex](z1) -- (z3);
\draw[-latex](z1) -- (z4);
\draw[-latex](z4) -- (z1);
\draw[-latex](z1) -- (z5);
\draw[-latex](z5) -- (z1);
\draw[-latex](z2) -- (z4);
\draw[-latex](z2) -- (z5);
\draw[-latex](z4) -- (z2);
\draw[-latex](z5) -- (z2);
\draw[-latex](z2) -- (z6);

\end{tikzpicture}
\caption{An example of dependency digraph. }\label{fig:dep-digraph}
\end{figure}

Note that $D$ may contain bidirectional edges, which can be viewed as directed cycles of length two; see~\cref{fig:dep-digraph} for an example.  To clarify, the induced coloring inherits the vertex ordering as well as the colors of the edges.

The main result of this section is the following embedding result which relates the problem of upper bounding $g(k)$ to the Ramsey problem for $f(k)$.

\begin{theorem} \label{thm:paths-embedding}
    For any integer $k \ge 1,$ and real $\eta > 0,$ there exists $c > 0$ such that the following holds. 
    Let $G = G_<$ be an ordered graph with no induced monotone $P_{2k}$. If $\overline{G}$ contains at least $\eta n^{f(k)}$ copies of $K_{f(k)}$, then $\overline{G}$ contains a $K_2[\frac{cn}{\log n}]$. In other words,
    $$g(k)\le f(k).$$
\end{theorem}

Let us first derive that $g(3)=f(3)=3$ and $g(4)=f(4)=5$, which imply~\cref{thm:p5,thm:p7}.
Recall that the optimal construction for $\beta(5)=1/2$ infers that $g(3)>2$ and the construction in~\cref{thm:induced-intro} shows that $g(4)>4$. Thus, by~\cref{thm:paths-embedding}, it suffices to prove $f(3)\le 3$ and $f(4)\le 5$.

\begin{cor}
    We have $f(3)\le 3$ and $f(4)\le 5$. Consequently, for any $\eps > 0$ and $n$-vertex ordered graph $G$ with no induced monotone $P_{6}$ (or $P_8$ resp.), if $\varrho_3(\overline{G})>\eps$ (or $\varrho_5(\overline{G})>\eps$), then $\overline{G}$ contains a $K_2[\frac{cn}{\log n}]$, where $c=c(\eps)>0$.
\end{cor}
\begin{proof}
    To prove $f(3)\le 3$, we need to show that any red-blue coloring $\chi$ of $\binom{[3]}{2}$ is admissible. We may assume that $D(\chi)$ has at least one edge and thus by symmetry we may assume that $\chi(1,2) = \mathrm{red}$ and $\chi(1, 3) = \chi(2, 3) = \mathrm{blue}$. But then, $E(D(\chi)) = \{ (3, 1), (3, 2) \},$ which is acyclic. Thus $\chi$ is admissible as desired.

    To prove $f(4)\le 5$, let $\chi$ be a red-blue coloring of $\binom{[5]}{2}$, we need to show that the induced coloring on some set of $4$ vertices is admissible. Suppose $D(\chi|_{1234})$ is not acyclic, for otherwise we are done. We claim that  $\chi(12)=\chi(34)=\mathrm{red}$. To see this, suppose one of them is blue and by symmetry we may assume that $\chi(34)=\mathrm{blue}$. Now if $\chi(23)=\mathrm{red}$, then $D(\chi|_{1234})$ contains only edges from $\{1,4\}$ to $\{2,3\}$ if $\chi(12)=\mathrm{blue}$ or only edges from $\{4\}$ to $\{1,2,3\}$ if $\chi(12)=\mathrm{red}$, both of which are acyclic. Thus we may assume $\chi(23)=\mathrm{blue}$, and then $D(\chi|_{1234})$ contains only edges from $\{3,4\}$ to $\{1,2\}$, which is acyclic. 

    The above analysis shows that if none of $D(\chi|_{1234}), D(\chi|_{2345}), D(\chi|_{1345})$ is acyclic, then $\chi(12)=\chi(34)=\chi(23)=\chi(13)=\mathrm{red}$ and so $D(\chi|_{1234})$ contains only the edge from $\{4\}$ to $\{1,2\}$, which is acyclic, a contradiction.
\end{proof}

We now prove an upper bound for the general case, which together with~\cref{thm:paths-embedding}, implies the upper bound for $g(k)$ in~\cref{thm:induced-intro}.

\begin{prop}
    We have $f(k)\le \frac{k^2-k+2}{2}$.
\end{prop}
\begin{proof}
 Suppose for a contradiction that there exists a red-blue edge-coloring $\chi$ of the ordered clique on vertices $1,2,\cdots,\frac{k^2-k+2}{2}$ such that the coloring induced on any $k$ vertices is not admissible. If there exists $t$ vertices $i_1< \cdots < i_{t}$ such that $\chi(i_ri_{r+1})=\mathrm{blue}$ for any $1\le r\le t-1$, we say these $t$ vertices form a blue path with order $t$.  For any vertices $v_i$, let $\sigma(v_i)$ be the order of the longest blue path starting from $v_i$.  If there are no blue edge starting from $v_i$, set $\sigma(v_i)=1$. If there exists a blue path with order $k$, then the coloring induced on these $k$ vertices is admissible as $D(\chi)$ is the empty graph, a contradiction.  Let $S_t=\{v_i|\sigma(v_i)=t\}$ for any $t\le k-1$, then these $S_t$ form a partition of the vertices.
\begin{claim}
 For $t\le k-1$, $|S_t|\le k-t$.
\end{claim}
\begin{poc}
Suppose $|S_t|\ge k+1-t$, by the definition of $S_t$, we can pick vertices $i_1<i_2<\cdots<i_{k+1-t}$ from $S_t$ and a blue order-$t$ path on vertices $i_{k+1-t}<\cdots<i_{k}$. Note that $\chi(i_ui_v)=\mathrm{red}$ for any $1\le u< v\le k+1-t$, as otherwise adding $i_u$ to the blue path of order $t$ starting from $i_v$ yields a blue path of order $t+1$, a contradiction to $i_v\in S_t$.  Let $D$ be the dependency digraph induced on the $k$ vertices $i_1,\ldots,i_k$. By the definition of $D$, for any directed edge $(u,v)\in D$, we must have $u> k+1-t$ and $v\le k+1-t$. Thus, $D$ is acyclic, a contradiction.
\end{poc}
Hence, $\frac{k^2-k+2}{2}=\sum_{t\in[k-1]}|S_t|\le \sum_{t\in[k-1]}(k-t)=\frac{k^2-k}{2},$ a contradiction.
\end{proof}

\subsection{Proof of Theorem~\ref{thm:paths-embedding}}
    By taking sufficiently small $c$, we may assume that $n$ is sufficiently large in terms of $k$ and $\eta.$ We define additional constants $c, \alpha, \beta, \gamma, \gamma_1, \gamma_2, \varepsilon, \varepsilon_1$, such that
    \[ 0 < 1/n \ll c \ll \alpha \ll \gamma=\gamma_{1}\gamma_{2} < \gamma_1 \ll \varepsilon_1 \ll \gamma_2 \ll \varepsilon \ll \beta \ll \eta<1/k. \]
    Suppose, for the sake of contradiction, that $\overline{G}$ does not contain a copy of $K_2[\frac{cn}{\log n}]$.
    \begin{claim}
        $G$ is $(\alpha, \beta)$-dense.
    \end{claim}
    \begin{poc}
        Suppose otherwise, so there is a set $S \subseteq V(G)$ with $|S| \ge \alpha n$ and $e_G(S) < \beta \binom{|S|}{2}.$ Removing vertices of degree more than $4 \beta |S|$ in $G[S],$ we obtain a set $S'$ of size at least $|S| / 2 \ge \alpha n / 2$ such that $\Delta(G[S']) < 8 \beta |S'|.$ Then~\cref{thm:tomon-pach-ordered} implies the existence of $K_2[\frac{cn}{\log n}]$ in $\overline{G[S']}$, contradicting our assumption.
    \end{poc}
    
    By Lemma~\ref{lem:clique-in-reduced-graph}, there exist sets $A_1, \dots, A_{f(k)} \subseteq V(G)$ such that $A_1 < A_2 < \dots < A_{f(k)},$ for all $i \in [f(k)], |A_i| \ge \gamma_1 n$ and for all $1 \le i < j \le f(k),$ $(A_i, A_j)$ is $\varepsilon_1$-regular with density at least $\frac{\eta}{2}$ in $\overline{G}.$

    We define an edge coloring $\chi_0$ of the complete graph on the ordered vertex set $\{1,2,\ldots, f(k)\}$ as follows. For $1 \le i < j \le f(k)$, let
    \[ \chi_0(ij) = \begin{cases}
        \mathrm{blue}, &\text{if } d(A_i, A_j) \ge \beta^2\\
        \mathrm{red}, &\text{otherwise}.
    \end{cases}
    \]

    By definition of $f(k),$ there exist indices $a_1 < \dots < a_k$ such that $\chi_0[\{a_1, \dots, a_k\}]$ is admissible. For $i \in [k],$ let $B_i = A_{a_i}.$ Moreover, define a coloring $\chi \colon \binom{[k]}{2} \rightarrow \{ \mathrm{red}, \mathrm{blue} \}$ by setting $\chi(ij) = \chi_0(a_ia_j)$ and note that $\chi$ is admissible. We now show that we can either construct an induced monotone $P_{2k}$ in $G$ by placing two consecutive vertices in each set $B_i$ or there is a biclique of size $\frac{cn}{\log n}$ in $\overline{G},$ thus obtaining a contradiction.
    
    Recall that $|B_i| \ge \gamma_1 n \ge \alpha n,$ so since $G$ is $(\alpha, \beta)$-dense, we have $d(G[B_i]) \ge \beta.$ Applying Lemma~\ref{lem:clique-in-reduced-graph} to $G[B_i]$ for each $i \in [k],$ we obtain two sets $C_{2i-1}, C_{2i} \subseteq B_i$ of size at least $\gamma_2 |B_i| \ge \gamma_2 \gamma_1 n=\gamma n$ with $C_{2i-1} < C_{2i}$ and such that $(C_{2i-1}, C_{2i})$ is $\varepsilon$-regular in $G$ with density at least $\frac{\beta}{4}$. To summarize, we have that $C_1 < C_2 < \dots< C_{2k},$ $|C_i| \ge \gamma_2 |B_i| \ge \gamma n,$ for all $i \in [2k].$ Moreover, using that $\frac{\varepsilon_{1}}{\gamma_{2}}< \varepsilon$, by Lemma~\ref{lem:slicing}, the pairs $(C_i, C_j)$ are $\varepsilon$-regular. Finally, the following holds for the densities between different pairs. Let $1 \le i < j \le 2k,$ then

    \begin{enumerate}[label=\alph*)]
        \item \label{dense-consecutive} If $\ceil{\frac{i}{2}} = \ceil{\frac{j}{2}}$, then $d(C_i, C_j) > \frac{\beta}{4}$.
        \item \label{not-too-dense} If $\ceil{\frac{i}{2}} \neq \ceil{\frac{j}{2}}$, then $d(C_i, C_j) < 1 - \frac{\eta}{4}$.
        \item \label{density-red} If $\chi(\ceil{\frac{i}{2}}, \ceil{\frac{j}{2}}) = \mathrm{red}$, then $d(C_i, C_j) < 2 \beta^2$.
        \item \label{density-blue} If $\chi(\ceil{\frac{i}{2}}, \ceil{\frac{j}{2}}) = \mathrm{blue}$, then $d(C_i, C_j) > \frac{\beta^{2}}{2}$.
    \end{enumerate}    

    Let $D = D(\chi)$ be the dependency digraph of $\chi$ as defined in~\cref{def:dependency-digraph}. By assumption, $D$ is acyclic. Consider a fixed topological ordering of $D$ and for $i \in [k]$, let $\pi(i)$ denote the position of vertex $i$ in this topological ordering. In other words, we have for every edge $ij \in E(D)$, $\pi(i) < \pi(j)$.    

    Let $(i_1, i_1+1), (i_2, i_2+1), \dots, (i_{k-1}, i_{k-1}+1)$ be an ordering of the edges $(1, 2), \dots, (k-1, k)$ according to the earliest appearance in $\pi$ of one of their endpoints. In other words for $j < j',$ we have that $\min\{\pi(i_j), \pi(i_j+1)\} \le \min\{\pi(i_{j'}), \pi(i_{j'}+1)\}$.

    We shall construct a monotone induced path $x_1, \dots, x_{2k}$ with $x_j \in C_j$ for all $j \in [2k].$ We first embed all but $x_1$ and $x_{2k}$ in $k-1$ steps. At step $t \in [k-1],$  we shall define the vertices $x_{2i_t}, x_{2i_t+1}.$ Finally, we embed $x_1$ and $x_{2k}.$

    Let us now formally describe the embedding procedure. For $t \in [k],$ after $t-1$ steps we shall have the following. Let $I_t \subseteq [2k]$ denote the set of vertices defined after $t-1$ steps. Thus, for $t \in [k],$ we will always have
    \[ I_t = \bigcup_{j=1}^{t-1} \{ 2i_j, 2i_j+1 \}. \]
    For each $i \in [2k] \setminus I_t$ (so $x_i$ has not been embedded yet), we will have a set $C^t_i$ of size at least $\beta^{8(t-1)} \cdot |C_i|$ into which we can embedded vertex $x_i$. Formally, after $t-1$ steps for $i \in [2k]$, we define

    \[ C_i^t = \left( C_i \cap \bigcap_{j \in I_t \cap \{ i-1, i+1\}} N(x_j) \right) \setminus \bigcup_{j \in I_t \setminus \{ i-1, i+1\} } N(x_j). \]
    We initialize with $C_i^1 = C_i$ for all $i \in [2k]$.
    
    Given that we have run the procedure for $t-1 \le k-2$ steps we show we can run step $t.$ Let $i = 2i_t$ and recall that we aim to embed vertices $x_i$ and $x_{i+1}.$ By assumption, we have sets $C^t_i$ and $C^t_{i+1}$ of size at least $\beta^{8(t-1)} |C_i|$ into which we can embed the vertices $x_i$ and $x_{i+1}$, respectively. Note that 
    \[ \frac{|C^t_i|}{|B_i|}= \frac{|C^t_i|}{|C_i|}\cdot\frac{|C_i|}{|B_i|} \ge \beta^{8(t-1)} \gamma_2 > \varepsilon_1^{0.1}, \]
which implies that for all $i, j \in [2k] \setminus I_t, i \neq j,$ the pair $(C^t_i, C^t_j)$ is $\varepsilon^{0.9}$-regular. Hence, there is a subset $C_i' \subseteq C^t_i$ of size at least \[|C^t_i|(1 - 2k \varepsilon^{0.9}) \ge \frac{|C^t_i|}{2} \] such that for every $v \in C_i'$ and $j \in [2k] \setminus (I_t \cup \{i, i+1\}),$ 
    \[\left| d(v,C^t_j) - d(C_i, C_j)\right| < \sqrt{\varepsilon}.\tag{$\clubsuit$} \] Analogously, there is a set $C_{i+1}' \subseteq C^t_{i+1}$ of size at least $\frac{|C^{t}_{i+1}|}{2}$ such that for all $v \in C_i'$ and $j \in [2k] \setminus (I_t \cup \{i, i+1\}),$ $\left| d(v,C^t_j) - d(C_{i+1}, C_j)\right| < \sqrt{\varepsilon}$.
    
    Now we consider two cases.

\medskip

\noindent\textbf{Case 1.} $\chi(i_t, i_t+1) = \mathrm{blue}$.
    
    Note that $\ceil{\frac{i}{2}} = i_t$ and $\ceil{\frac{i+1}{2}} = i_t+1$. By condition \ref{density-blue}, we have $d(C_i, C_{i+1}) > \frac{\beta^{2}}{2}$ and thus by $\varepsilon$-regularity and choices of parameters, $d(C'_i, C'_{i+1}) > \frac{\beta^{2}}{2} - \sqrt{\varepsilon} > \beta^3.$ In particular, there is a vertex $u \in C'_i$ with $|N(u) \cap C'_{i+1}| \ge \frac{\beta^{3}}{2} \cdot |C'_{i+1}|$. Pick such a vertex $u.$ For $j \in [2k] \setminus (I_t \cup \{i-1, i, i+1\})$, let
    $C'_j = C^t_j \setminus N(u).$  Furthermore, define $C'_{i-1} = C^t_{i-1} \cap N_G(u)$ if $i-1 \not\in I_t$.  
    
    Recall that $d(C_i, C_j) < 1 - \frac{\eta}{4}$ for $j \in [2k] \setminus (I_t \cup \{i-1, i, i+1\})$ by condition \ref{not-too-dense} and $d(C_{i-1}, C_i) > \frac{\beta}{4}$ by condition \ref{dense-consecutive}. Therefore, for all $j \in [2k] \setminus (I_t \cup \{i, i+1\}),$ combining with  $(\clubsuit)$ we have \[|C'_j| \ge \frac{\beta^{3}}{100}\cdot |C^t_j| \ge \beta^{8(t-1) + 4} \cdot \gamma n.\]
    
    By $\varepsilon$-regularity, there is a vertex $v \in N(u) \cap C'_{i+1}$ such that for all $j \in [2k] \setminus (I_t \cup \{i, i+1\}),$ we have \[\left|d(v,C'_j) - d(C_{i+1}, C_j)\right| < \sqrt{\varepsilon}.\] Pick such a vertex $v$ and let $x_i = u$ and $x_{i+1} = v.$ Recalling that $d(C_{i+1}, C_j) < 1 - \frac{\eta}{4}$ for $j \neq i+2$ and $d(C_{i+1}, C_{i+2}) > \frac{\beta}{4}$, it follows that $|C^{t+1}_j| \ge \beta^{8t} \cdot |C^t_j|$ for all $j \in [2k] \setminus I_{t+1}$ as needed.

\medskip

\noindent\textbf{Case 2.} $\chi(i_t, i_t+1) = \mathrm{red}$.

    Since $|C'_i|, |C'_{i+1}| \ge\frac{1}{2}\cdot\beta^{8(t-1)}\gamma n> \frac{n}{\log n} $ and $\overline{G}$ is $K_{2}[\frac{cn}{\log n}]$-free, there exists an edge $uv \in E(G)$ with $u \in C'_i, v \in C'_{i+1}.$ Set $x_i = u, x_{i+1} = v.$ We will show that for all $j \in [2k] \setminus I_{t+1},$ $|C^{t+1}_j| \ge \beta^8 |C^t_j|$, which is sufficient by our assumption. 
    
    Consider first $j = i-1.$ If $j \not\in [2k] \setminus I_t,$ there is nothing to prove. Otherwise, $d_G(C_j, C_i) >\frac{\beta}{4}$ by condition~\ref{dense-consecutive} and $d(C_j, C_{i+1}) < 2 \beta^2$ by condition~\ref{density-red}. Hence, by the definitions of the sets $C'_i, C'_{i+1},$ we have \[|C^{t+1}_j| = |C^t_j \cap N(u) \setminus N(v)| \ge |C^t_j|\cdot\bigg(\frac{\beta}{4} - \sqrt{\varepsilon} - 2\beta^2 - \sqrt{\varepsilon}\bigg) \ge \beta^8 |C_t^j|.\]
    
    Completely analogously, if $i+2 \in [2k] \setminus I_{t},$ we have $|C^{t+1}_{i+2}| \ge \beta^8 |C^t_{i+2}|.$

    Now, consider $j \in [2k] \setminus (I_{t} \cup \{i-1, i+2\})$. If $\chi(\floor{\frac{i}{2}}, \floor{\frac{j}{2}}) = \mathrm{red}$ or $\chi(\floor{\frac{i+1}{2}}, \floor{\frac{j}{2}}) = \mathrm{red},$ then $d(C_i, C_j) + d(C_{i+1}, C_j) < 2\beta^2 + 1 - \frac{\eta}{4}$ and so combining with $(\clubsuit)$, we have \[|C^{t+1}_j| \ge |C^t_j|\cdot\bigg(1 - \frac{\eta}{4} - 2\beta^2 - 2\sqrt{\varepsilon}\bigg) > \beta^8 |C^t_j|,\] as needed. Then we can assume that $\chi(\floor{\frac{i}{2}}, \floor{\frac{j}{2}}) = \chi(\floor{\frac{i+1}{2}}, \floor{\frac{j}{2}}) = \mathrm{blue}$. Then by definition, we can see that $(\floor{\frac{j}{2}}, \floor{\frac{i}{2}}), (\floor{\frac{j}{2}}, \floor{\frac{i+1}{2}})\in E(D)$, which implies that $\pi(\floor{\frac{j}{2}}) < \pi(\floor{\frac{i}{2}}), \pi(\floor{\frac{i+1}{2}})$. Therefore, $(\floor{\frac{j}{2}}, \floor{\frac{j}{2}}+1)$ appears before $\floor{\frac{i}{2}}, \floor{\frac{i}{2}}+1$ in the sequence $(i_1, i_1+1), \dots, (i_{k-1}, i_{k-1}+1)$ which in turn implies that $j \in I_t,$ contradicting our assumption.
    
    We have shown that we can run all $k-1$ steps of the above procedure. Finally, we have two candidate sets $C^k_1, C^k_{2k}$ of size at least $\beta^{8k} \gamma_2 |B_1|$ into which we can embed the vertices $x_1$ and $x_{2k}$ respectively. Recalling that $d(C_1, C_{2k}) < 1 - \frac{\eta}{4}$, by $\varepsilon_1$-regularity, there are vertices $u \in C^k_1, v \in C^k_{2k}$ such that $uv \not\in E(G).$ Setting $x_1 = u, x_{2k} = v,$ we have that $x_1, x_2, \dots, x_{2k}$ is a monotone induced path in $G$, contradicting our assumption.
    
    This finishes the proof of~\cref{thm:paths-embedding}.

\section{Lower bound construction for Theorem~\ref{thm:induced-intro}}\label{section:ConstructionBEGraph}

In this section, we proceed to construct $n$-vertex ordered graphs $G$ with no induced monotone path $P_{2k-1}$ such that its complement $\overline{G}$ has positive $K_{\lfloor\frac{k^{2}}{4}\rfloor}$-density and no biclique of size $n/e^{o(\sqrt{\log n/\log\log n})}$.

Now let $\eps>0$, $n$ be sufficiently large, $h=\frac{\log{n}}{\log{\log{n}}}$ and $\mu=\frac{\varepsilon}{\sqrt{h}}$. Set $s=\floor{\frac{k}{2}}$ and $t=\ceil{\frac{k}{2}}$. For simplicity, we only provide the proof details for even $k$ so that the floors/ceilings can be dropped; the argument for odd $k$ is identical.

\medskip

\noindent\textbf{Construction.} By~\cref{lem: partition-sphere}, as $n$ is sufficiently large, we can partition the unit sphere $\mathbb{S}^{h-1}$ into $\frac{2n}{k}$ many pieces $D_{1},D_{2},\ldots,D_{\frac{2n}{k}}$ of equal measure such that each piece has diameter at most $\frac{\mu}{4}$. We construct $V(G)=V_{1}\sqcup V_{2}\sqcup\cdots\sqcup V_{s}$ as follows. For each $i\in[s]$, we select an arbitrary point from each $D_{j}$, $j\in[\frac{2n}{k}]$, as a vertex. So each $V_{i}$ consists of $\frac{2n}{k}$ vertices. We randomly partition each $V_{i}=V_{i}^{(1)}\sqcup V_{i}^{(2)}\sqcup\cdots\sqcup V_{i}^{(t)}$ into $t$ parts of equal size. Moreover, we label the vertices so that $V_{1}<V_{2}<\cdots<V_{s}$ and for each $i\in [s]$, $V_{i}^{(1)}<V_{i}^{(2)}<\cdots<V_{i}^{(t)}$. For convenience, for a pair of vertices $x,y\in V(G)$, we write $|x-y|$ for the distance between the corresponding points on the sphere.

The edge set of $G$ consists of the following three types of edges; see~\cref{fig:graph-on-sphere}.
\begin{enumerate}
    \item For each $1\le i\le s$ and $1\le j\le t$, $G[V_{i}^{(j)}]$ is a clique.
    \item For given $1\le i\le s$ and $j_{1}\neq j_{2}$, $x\in V_{i}^{(j_{1})}$ and $y\in V_{i}^{(j_{2})}$ are adjacent if $|x-y|\ge 2-\mu$.
    \item For given $i_{1}\neq i_{2}$, $x\in V_{i_{1}}$ and $y\in V_{i_{2}}$ are adjacent if $|x-y|> \sqrt{2}-\mu$.
\end{enumerate}

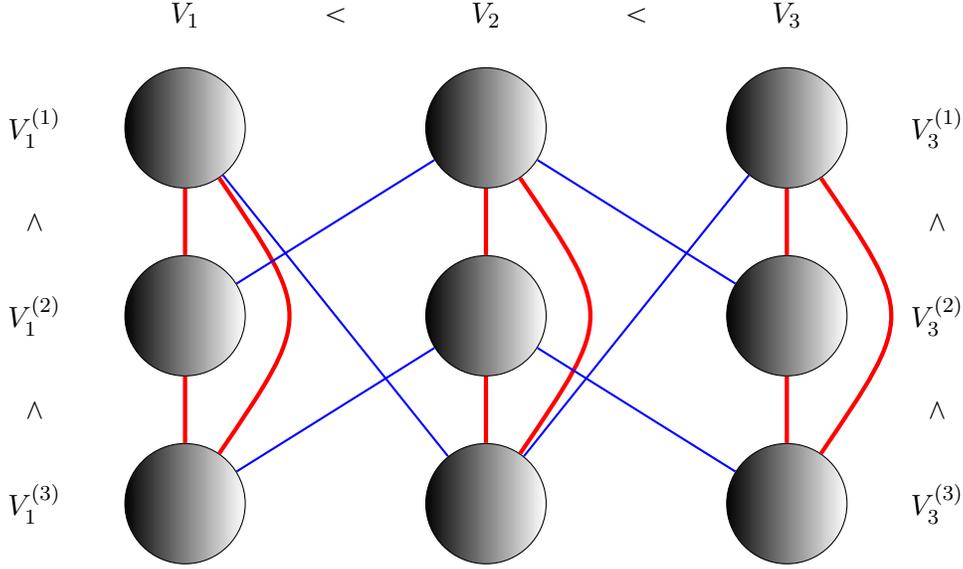
\begin{figure}[bhtp]
    \centering
\begin{tikzpicture}

\node[circle, draw, minimum size=1.6cm, left color=black, right color=white] (A1) at (0,5) {};

\node[circle, draw, minimum size=1.6cm, left color=black, right color=white] (A2) at (0,2.5) {};

\node[circle, draw, minimum size=1.6cm, left color=black, right color=white] (A3) at (0,0) {};

\draw[red, ultra thick] (A1) -- (A2);

\draw[red, ultra thick] (A2) -- (A3);
 \draw[red, ultra thick] (A1) .. controls (1.7,2.5) .. (A3);

\node[circle, draw, minimum size=1.6cm, left color=black, right color=white] (B1) at (4,5) {};

\node[circle, draw, minimum size=1.6cm, left color=black, right color=white] (B2) at (4,2.5) {};

\node[circle, draw, minimum size=1.6cm, left color=black, right color=white] (B3) at (4,0) {};

\draw[red, ultra thick] (B1) -- (B2);

\draw[red, ultra thick] (B2) -- (B3);

 \draw[red, ultra thick] (B1) .. controls (5.7,2.5) .. (B3);

\node[circle, draw, minimum size=1.6cm, left color=black, right color=white] (C1) at (8,5) {};

\node[circle, draw, minimum size=1.6cm, left color=black, right color=white] (C2) at (8,2.5) {};

\node[circle, draw, minimum size=1.6cm, left color=black, right color=white] (C3) at (8,0) {};

\draw[red, ultra thick] (C1) -- (C2);

\draw[red, ultra thick] (C2) -- (C3);

 \draw[red, ultra thick] (C1) .. controls (9.7,2.5) .. (C3);

\node[black] at (0, 6.5) {$V_{1}$};
\node[black] at (2, 6.5) {$<$};
\node[black] at (4, 6.5) {$V_{2}$};
\node[black] at (6, 6.5) {$<$};
\node[black] at (8, 6.5) {$V_{3}$};
\node[black] at (-2,3.75) {\rotatebox{90}{$>$}};
\node[black] at (-2,1.25) {\rotatebox{90}{$>$}};
\node[black] at (10,3.75) {\rotatebox{90}{$>$}};
\node[black] at (10,1.25) {\rotatebox{90}{$>$}};
\node[black] at (-2, 5) {$V_{1}^{(1)}$};
\node[black] at (-2, 2.5) {$V_{1}^{(2)}$};
\node[black] at (-2, 0) {$V_{1}^{(3)}$};

\node[black] at (10, 5) {$V_{3}^{(1)}$};
\node[black] at (10, 2.5) {$V_{3}^{(2)}$};
\node[black] at (10, 0) {$V_{3}^{(3)}$};


\draw[blue, thick] (A1) -- (B3);

\draw[blue, thick] (A2) -- (B1);

\draw[blue, thick] (A3) -- (B2);

\draw[blue, thick] (B1) -- (C2);

\draw[blue, thick] (B2) -- (C3);

\draw[blue, thick] (B3) -- (C1);

\end{tikzpicture}
    \caption{The black shading represents a clique and is of type-1, The red and blue edges are of type-2 and 3 respectively.}
    \label{fig:graph-on-sphere}
\end{figure}

\begin{claim}
    $G$ does not contain  induced monotone $P_{2k-1}$.
\end{claim}
\begin{poc}
    Suppose $P$ is an induced monotone $P_{2k-1}$ on vertices $a_{1}<a_{2}<\cdots<a_{2k-1}$. Note first that as each $V_i^{(j)}$ induces a clique, it contains at most two vertices of the induced path $P$. Thus $P$ has at most $2\cdot t=k$ vertices in each $V_i$, $i\in[s]$. Consequently, if all but at most one $p_i=|V(P)\cap V_i|$, $i\in [s]$, are at most 2, then $|P|=\sum_{i\in[s]}p_i\le k+2(s-1)=2k-2$, a contradiction. We may then assume that at least two $p_i$ are at least 3. Note that as $V_{1}<V_{2}<\cdots<V_{s}$, each $V_i$ must contain consecutive vertices in $P$. 

    By definitions of type-2 and type-3 edges, there cannot be two type-2 edges $aa'\in V_i$ and $bb'\in V_j$ with $i\neq j$ such that $ab,ab',a'b,a'b'\not\in E(G)$ due to~\cref{thm:BEIsK4free}.

    We claim that if any $p_i\ge 3$, then $P[V_i]$ contains a type-2 edge. To see this, say $a_r,a_{r+1},a_{r+2} \in P\cap V_i$. Then these three vertices cannot be in the same $V_{i}^{(j)}$ for any $j\in[t]$ as $V_i^{(j)}$ induces a clique. Then at least one of the edges $a_{r}a_{r+1}$ and $a_{r+1}a_{r+2}$ is of type-2.

    If there are at least three say $p_i,p_j,p_{\ell}\ge 3$ with $i<j<\ell$. Then by the above discussion, there exist type-2 edges $e$ in $P[V_i]$ and $e'$ in $P[V_{\ell}]$. Moreover, as $p_j>0$, there are no edges between these type-2 edges $e$ and $e'$, contradicting~\cref{thm:BEIsK4free}. 

    Thus, we may assume that exactly two, say $V_1, V_2$, containing at least $3$ vertices of $P$, and the remaining parts contain at most $2$ vertices each. If $p_1=p_2=3$, then $|P|\le 3+3+2(s-2)\le k+2\le 2k-2$ as $k\ge 4$. Thus, without loss of generality $p_1\ge 4$ and so $a_1,a_2,a_3,a_4\in P[V_1]$. Then $\{a_1,a_2,a_3\}$ and $P[V_2]$ both contain a type-2 edge and there is no edge between these two edges because of $a_4$, again contradicting~\cref{thm:BEIsK4free}.
\end{poc}

Next, we verify that there is no large biclique in $\overline{G}$.

\begin{claim}
    $\overline{G}$ does not contain $K_{2}[\frac{cn}{\log{n}}]$ as a subgraph.
\end{claim}
\begin{poc}
    Suppose there is a copy of $K_{2}[\frac{cn}{\log{n}}]$ in $\overline{G}$. By pigeonhole principle, there is a copy of $K_{2}[\frac{cn}{k\log{n}}]$ in $\overline{G}$ with parts $S\subseteq V_i$ and $T\subseteq V_j$.
    
    If $i=j$, then by pigeonhole principle again, there exist $a,b\in [t]$ such that $\overline{G}$ contains a copy of $K_{2}[\frac{2cn}{k^{2}\log{n}}]$ with parts $S_{a}\subseteq V_{i}^{(a)}$ and $T_{b}\subseteq V_{i}^{(b)}$. Note that the Lebesgue measures of the pieces $D_j$, $j\in[\frac{2n}{k}]$, corresponding to both $S_{a}$ and $T_{b}$ are at least $\frac{2cn}{k^{2}\log{n}}\cdot \frac{k}{2n}>e^{\frac{-h\mu}{4}}$, therefore, by~\cref{lemma:LargeMeasureToLargeDistance}, $d_{\max}(S_{a},T_{b})\ge 2-\frac{\mu}{2}$. This, together with the fact that the diameter of each piece $D_j$ is at most $\mu/4$, implies that there must be a type-2 edge $uv$ with $u\in S_{a}$, $v\in T_{b}$, a contradiction. 
    
    If $i\neq j$, then both of the Lebesgue measures of $S$ and $T$ are at least $\frac{cn}{k\log{n}}\cdot \frac{k}{2n}>e^{\frac{-h\mu}{4}}$. Therefore the diameters of $S$ and $T$ are at least $2-\frac{\mu}{2}$. We can then find $x,x'\in S\subseteq V_{i}$ and $y,y'\in T\subseteq V_{j}$ with $|x-x'|,|y-y'|\ge 2-\mu$. Moreover, by definition of $G$, all of pairs $xy,xy',x'y,x'y'\notin E(G)$ implies that all of the distances $|x-y|$, $|x-y'|$, $|x'-y|$ and $|x'-y'|$ are at most $\sqrt{2}-\mu$, contradicting~\cref{thm:BEIsK4free}.
\end{poc}

\begin{claim}
   The expected number of $\frac{k^{2}}{4}$-cliques in $\overline{G}$ is at least $\eta n^{\frac{k^{2}}{4}}$ for some $\eta>0$.
\end{claim}
\begin{poc}
  The number of $\frac{k^{2}}{4}$-cliques in $\overline{G}$ is at least the number of $\frac{k^{2}}{4}$-tuples $(v_{i}^{(j)}:i\in [s],j\in [t])$ with $v_{i}^{(j)}\in V_{i}^{(j)}$ such that the following conditions hold.
  \begin{itemize}
      \item For each $1\le i\le s$, we have  $|v_{i}^{(j_{1})}-v_{i}^{(j_{2})}|<2-\mu$ for any $1\le j_{1}<j_{2}\le t$.
      \item For each $1\le i_{1}<i_{2}\le s$, we have $|v_{i_{1}}^{(j_{1})}-v_{i_{2}}^{(j_{2})}|\le \sqrt{2}-\mu$ for any $j_{1},j_{2}\in [t]$.
  \end{itemize}
  For each $1\le i\le s$ and $1\le j\le t$, take an arbitrary vertex $a_{i}^{(j)}\in V_{i}^{(j)}$. As $V_{i}^{(j)}$, $j\in[t]$, is a random partition of $V_i$, we have by~\cref{lemma:Measure} that for every $1\le i_{1}<i_{2}\le s$ and $j_{1},j_{2}\in [t]$,
  \begin{equation*}
      \mathbb{P}\bigg[\big|a_{i_{1}}^{(j_{1})}-a_{i_{2}}^{(j_{2})}\big|\le \sqrt{2}-\mu\bigg]\ge\frac{1}{2}-\sqrt{2}\eps.
  \end{equation*}
  
  Moreover, for a pair of vertices $a_{i}^{(j_{1})}$ and $a_{i}^{(j_{2})}$ with $j_1\neq j_2$, $a_{i}^{(j_{1})}a_{i}^{(j_{2})}\in E(G)$ implies $|a_{i}^{(j_{1})}-a_{i}^{(j_{2})}|\ge 2-\mu$. Note that for any vertex $v\in V_i$, its neighbors in $V_i$ all lie in a spherical cap $C$ whose points are within distance $d=d_0\pm \mu/4$ from its centre, where $d_0^2=4\mu-\mu^2$. As the height of such a cap is $d^2/2\le 4\mu$. Thus by~\cref{lem:cap-UB}, the measure of $C$ is at most $e^{-h(1-4\mu)^2/2}$. Consequently, 
  \begin{equation*}
      \mathbb{P}[|a_{i}^{(j_{1})}-a_{i}^{(j_{2})}|< 2-\mu]\ge 1-e^{-h(1-4\mu)^2/2}>1-\eps.
  \end{equation*}
  Therefore, the probability that the selected vertices forming a copy of $\frac{k^{2}}{4}$-clique is positive, then the claim follows.
\end{poc}
Therefore with positive probability, the random graph $G$ has positive $K_\frac{k^{2}}{4}$-density.
This completes the proof of~\cref{thm:induced-intro}.

\section{Results in posets}\label{sec:Comparable}

\subsection{Incomparability graphs}

We first show how~\cref{lemma:incomparability-partition} implies~\cref{thm:main2}, which we restate for convenience.

\thmincomparabilityrestate*

\begin{proof}[Proof of~\cref{thm:main2}]
    Let $G$ be an incomparability graph on $n$ vertices with at least $cn^k$ $k$-cliques and for the sake of contradiction, suppose $G$ does not contain any copy of $K_k[q],$ where $q = \frac{c^4 n}{10^8 k^2 \log n}.$ Let $\varepsilon = \frac{c}{4}$ and apply~\cref{lemma:incomparability-partition} to obtain a partition $V(G) = V_0 \sqcup V_1 \sqcup \cdots \sqcup V_m$. Note that $\frac{n}{m} \ge t=\frac{\varepsilon^7 n}{10^{11} k^5} \ge q$ when $n$ is sufficiently large. We then consider $k$-cliques in $G$ and list three possible cases as follows.
\begin{enumerate}
    \item There are at most $|V_0| n^{k-1} \le \varepsilon n^k$ $k$-cliques containing at least one vertex in $V_{0}$.
    \item The number of $k$-cliques in $G$ such that there are at least two vertices in the same part is at most $m t^2 n^{k-2} \le t n^{k-1} \le \varepsilon n^k$. 
    \item For each $k$-clique with all vertices belonging to distinct parts $V_{i_{1}},V_{i_{2}},\ldots,V_{i_{k}}$,  note that at least one pair of these parts is inhomogeneous, for otherwise there is a copy of $K_k[t]$ and thus $K_k[q]$ in $G$. Since there are at most $\varepsilon m^2$ inhomogeneous pairs, the number of $k$-cliques in $G$ with all vertices in distinct parts among $V_1, \ldots, V_m$ is at most $\varepsilon m^2 t^2 n^{k-2} \le \varepsilon n^k$.
\end{enumerate}
Therefore, $G$ contains at most $3 \varepsilon n^k < c n^k$ $k$-cliques, which is a contradiction to our assumption.
\end{proof}

We proceed to the proof of \cref{lemma:incomparability-partition}.
\begin{proof}[Proof of~\cref{lemma:incomparability-partition}]

Let $(P,\prec)$ be a partially ordered set on $n$ elements such that $G$ is the incomparability graph of $P$. Let $<'$ be an arbitrary linear extension of $\prec$, and let $x_1<'x_{2}<'\cdots<'x_{n}$ be the enumeration of the elements of $G$ by $<'$. Let $s = \frac{10}{\varepsilon}$ and $\ell=\frac{10k}{\varepsilon}.$ We equally partition $P$ into $s$ intervals $P_1,\ldots, P_s$, where $P_i = \big\{x_{\frac{(i-1)n}{s}+1},\ldots, x_{\frac{in}{s}}\big\}$ for each $1\le i\le s$.

\begin{claim}\label{claim:FindingChain}
    Every subset of $P$ of size $\frac{\varepsilon n}{2s}$ contains $\ell$ disjoint sets $A_1,A_2,\ldots,A_\ell$ satisfying $A_1\succ A_2\succ\cdots\succ A_{\ell}$ and $|A_i| = t$ for $1\le i\le \ell$.
\end{claim}
\begin{poc}
Let $S$ be an arbitrary subset of $P$ of size $\frac{\varepsilon n}{2s}$ and apply~\cref{lm:FHT} to the subposet of $P$ induced on $S.$ If the former conclusion of~\cref{lm:FHT} holds, we obtain $\ell$ sets $A_1 \succ A_2 \succ \dots \succ A_\ell$, each of size $\frac{|S|}{10^4 \ell^5} \ge \frac{\varepsilon^7 n}{10^{11} k^5} = t$ as needed. If the latter conclusion of~\cref{lm:FHT} holds, there are $\ell$ pairwise incomparable sets in $P$, each of size $\frac{|S|}{40 \ell^2 \log |S|} > \frac{\varepsilon^4 n}{10^5 k^2 \log n} = q$. This implies the existence of $K_k[q]$ in $G,$ contradicting our assumption and finishing the proof of the claim.
\end{poc}

Let $m_0 = \lceil \frac{(1 - \varepsilon)n}{st \ell} \rceil.$ Then, for each $i \in [s],$ iteratively using \cref{claim:FindingChain}, we can find $\ell m_0$ disjoint sets $\{B_{ij}^u\}_{j \in [m_0], u \in [\ell]}$ in $P_{i}$, such that $|B_{ij}^u|=t$ and $B_{ij}^1 \prec B_{ij}^2 \prec \dots \prec B_{ij}^\ell$ for all $i \in [s], j \in [m_0].$ Let $V_0 = P \setminus \bigcup_{i \in [s], j \in [m_0], u \in [\ell]} B_{ij}^u$ and $m = s m_0 \ell$, and we identify $V_1, \dots, V_m$ with the sets $\{B_{ij}^u\}_{i \in [s], j \in [m_0], u \in [\ell]}$. Note that $|V_0| = n - s m_0 \ell t \le \varepsilon n$ as needed. It remains to verify that at most $\varepsilon m^2$ pairs $(V_i, V_j)$ are inhomogeneous in $G,$ that is, there are at most $\varepsilon m^2$ pairs $((i, j, u), (i', j', u'))$ such that neither $B_{ij}^u \prec B_{i'j'}^{u'}$, nor $B_{i'j'}^{u'} \prec B_{ij}^u$, nor $B_{ij}^u$ and $B_{i'j'}^{u'}$ are incomparable.

Observe that the number of pairs $((i, j, u), (i', j', u'))$ with $i = i'$ is at most $s (m_0 \ell)^2 = \frac{m^{2}}{s}=\frac{\varepsilon m^{2}}{10}$. To upper bound the number of inhomogeneous pairs coming from different sets $P_i,P_{i'}$, we shall need the following claim.

  \begin{figure}[bhtp]
    \centering
   \begin{tikzpicture}
   
\draw[black, thick] (-4, 2) -- (-2.5, 2);
\draw[black, thick] (-1.75, 2) -- (-0.25, 2);
\draw[black, thick] (2, 2) -- (3.5, 2);
\draw[black, thick] (-2.25, 0) -- (-0.75, 0);
\draw[black, thick] (0.75, 0) -- (2.25, 0);

    \fill (0.5,2) circle (0.5pt);
    \fill (1.0,2) circle (0.5pt);
     \fill (1.5,2) circle (0.5pt);
\draw[red, ultra thick, ->] (-3.3, 2) .. controls (-0.2, 3) .. (2.8, 2); 

\draw[red, ultra thick, ->] (-1.75,0) .. controls (0, 0.5) .. (1.55, 0);

\draw[blue, ultra thick, ->] (-3.3, 2) -- (1.75, 0);
\draw[red, ultra thick, ->] (2.8, 2) -- (-1.75, 0);
\node[black] at (-5, 2) {$i$};
\node[black] at (-5, 0) {$i'$};
\node[black] at (2.5, 2.4) {$x$};
\node[black] at (-1.75, 0.5) {$y$};
\node[black] at (-5, 1) {$i< i'$};
\node[black] at (-3, 3) {$B_{ij}^{u}$};
\node[black] at (-2.0, 3) {$\prec$};
\node[black] at (-0.5, 3) {$\prec$};
\node[black] at (1.5, 3) {$\prec$};
\node[black] at (2.5, 3) {$B_{ij}^{v}$};
\node[black] at (-1.5, -0.5) {$B_{i'j'}^{v'}$};
\node[black] at (1.5, -0.5) {$B_{i'j'}^{u'}$};
\node[black] at (0, -0.5) {$\prec$};
\end{tikzpicture}
    \caption{\cref{claim:incomparable-pairs}}
    \label{fig:Claim5.2}
\end{figure}

\begin{claim} \label{claim:incomparable-pairs}
    For fixed pairs $(i,j), (i',j')$ with $i < i'$, the number of pairs $(u, u') \in [\ell]^2$ such that $(B_{ij}^u, B_{i'j'}^{u'})$ is inhomogeneous is at most $2 \ell - 1.$
\end{claim}
Before proving the claim, we finish the proof of the lemma. The number of inhomogeneous pairs coming from different sets $P_i,P_{i'}$ is at most $s^2 m_0^2 (2\ell - 1) \le \frac{2m^{2}}{\ell} < \frac{\varepsilon m^{2}}{5}$. Thus the total number of inhomogeneous pairs is at most $\varepsilon m^2,$ as required.

We are left to prove \cref{claim:incomparable-pairs}.
    Let $Z$ denote the set of pairs $(u, u') \in [\ell]^2$ such that $(B_{ij}^u, B_{i'j'}^{u'})$ is inhomogeneous. We will show that for any distinct pairs $(u, u'), (v, v') \in Z$, it holds that 
    $$u + u' \neq v + v'.$$ 
    Indeed, this would prove the claim since trivially $u + u' \in [2, 2\ell]$ for any $(u, u') \in Z$.

    Consider two distinct pairs $(u, u'), (v, v') \in Z$ and suppose that $u + u' = v + v'$. Without loss of generality, assume that $u < v$ which then implies $v' < u'.$ Since $u < v,$ it follows that $B_{ij}^u \prec B_{ij}^v.$ Because the pair $(B_{ij}^v, B_{i'j'}^{v'})$ is not incomparable, there exist some $x \in B_{ij}^v, y \in B_{i'j'}^{v'}$ which are comparable in $P.$ Since $i < i', x \in P_i, y \in P_{i'}$ and $x <' y$ by the assumed linear ordering, it follows that $x \prec y.$ Finally, since $v' < u',$ we have $y \prec B_{i'j'}^{u'}$ and by transitivity it follows that $B_{ij}^u \prec x \prec y \prec B_{i'j'}^{u'}$ (see~\cref{fig:Claim5.2}), contradicting our assumption that $(u, u') \in Z.$
\end{proof}

\subsection{\texorpdfstring{$r$}{r}-comparability graphs}
In this subsection, we prove~\cref{thm:main1}, which we restate here.

\thmmainrestate*

\begin{proof}
First we prove the case $r=1$. For two elements $x,y$, we write $x\prec  y$ if $x\prec_i y$ for every $i\in[r]$. For a copy of $K_{2h-1}$ in $G$ with vertex set $\{x_{1},x_{2},\ldots,x_{2h-1}\}$, we can assume the order is $x_{1}\prec x_{2}\prec\cdots\prec x_{2h-1}$. By the pigeonhole principle, there exists an $(h-1)$-tuple $(x_2,x_4,\ldots,x_{2h-2})$ in $G$ such that there are at least $\varepsilon n^{h}$ choices of $h$-tuples $(x_1,x_3,\ldots,x_{2h-1})$ where $x_{1}x_{2}\cdots x_{2h-1}$ together form a copy of $K_{2h-1}$. Fix such an $(h-1)$-tuple $(x_{2},x_{4},\ldots,x_{2h-2})$, let $D_{1}:=\{x|x\prec x_2\}$, $D_{h}:=\{x|x_{2h-2}\prec x\}$, and $D_{i}:=\{x|x_{2i-2}\prec x\prec x_{2i}$\} for any $2\le i\le h-1$. Note that if $x_1 \cdots x_{2h-1}$ form a copy of $K_{2h-1}$, then $x_{2i-1}\in D_i$ by our above assumption, which also yields that $\prod_{i=1}^h|D_i|\ge \varepsilon n^{h}$. Obviously, $|D_{i}|\le n$, therefore $|D_i|\ge \varepsilon n$ holds for all $1\le i\le h$. Moreover, for each $1\le i<j\le h$ and for any pair of vertices $a\in D_{i}, b\in D_{j}$, we have $a\prec x_{2i}\prec b$, therefore $G[\bigcup_{i=1}^{h}D_{i}]$ contains a copy of $K_{h}[\varepsilon n]$.

Next we proceed to the proof of the general cases $r\ge 2$. Let $K$ be an arbitrary  $k$-clique in $G$. For each vertex $x\in V(K)$, let $p_i(x)$ be the maximum number $t$ such that there is a chain of the form $x=a_{1}\prec_{i}a_{2}\prec_{i}a_{3}\prec_{i}\cdots\prec_{i}a_{t}$, where $a_{1},a_{2},\ldots,a_{t}$ are in $V(K)$. Denote $p(x) = (p_1(x), \dots, p_r(x)).$ We first prove the following claim.

\begin{claim}\label{cl:EZ}
For a given $k$-clique $K$ and any distinct $x,y\in V(K)$, we have $p(x) \neq p(y).$
\end{claim}

\begin{poc}
As $x$ and $y$ are adjacent in $G$, they are comparable in one of the partial orders. Without loss of generality, assume that $x\prec_{1} y$. By definition of $p_{i}(y)$, we can find a chain of the form $y=a_{1}\prec_{1}a_{2}\prec_{1}\cdots\prec_{1}a_{p_{1}(y)}$, such that all the elements $a_{2},a_{3},\ldots,a_{p_{1}(y)}$ are in $V(K)$. However, then we have $x\prec_{1}y\prec_{1}a_{2}\cdots\prec_{1}a_{p_1(y)}$, which implies that $p_1(x)>p_1(y)$. Therefore, $p(x) \neq p(y)$.
\end{poc}

Let $G_{i}=G(\prec_i)$ be the comparability graph of $\prec_{i}$. For each $k$-clique $K$ in $G$, by definition we have $p_i(x) \ge 1,$ for all $x \in V(K), i \in [r]$. By~\cref{cl:EZ}, since $k > (2h-2)^r,$ there must exist $x \in V(K)$ and $i \in [r]$ such that $p_i(x) \ge 2h-1$. In particular, this implies that $G_i$ contains a $(2h-1)$-clique inside $K$. By the pigeonhole principle again, there exists some index $j\in [r]$ such that for at least $\frac{\varepsilon n^{k}}{r}$ many $k$-cliques $K$ in $G$, there is a $(2h-1)$-clique in $G_j$ inside $K$. On the other hand, each $(2h-1)$-clique in $G_j$ can be contained in at most $n^{k-(2h-1)}$ many $k$-cliques in $G$, which yields that totally there are at least $\frac{\varepsilon n^{2h-1}}{r}$ many cliques of size $2h-1$ in $G_{j}$. Applying the result in the case $r=1,$ we obtain a copy of $K_{h}[\frac{\varepsilon n}{r}]$ in $G_{j}$.
\end{proof}

\section{Hereditary graphs and \texorpdfstring{$B_k$}{Bk} property}\label{sec:Hereditary}
\subsection{\texorpdfstring{$B_2$}{B2} property implies \texorpdfstring{$B_k$}{Bk} property in hereditary graphs}

\begin{proof}[Proof of~\cref{thm:B2ImpliesBk}]
Assume that $\mathcal{G}$ is a  hereditary family which has the $B_2$ property with function $f$. For any $\varepsilon>0$, we set a function \[h(\varepsilon)=\delta\ll\varepsilon_0\ll\varepsilon_1\ll\varepsilon_2\ll\varepsilon_3\ll\varepsilon_4\ll\varepsilon.\] We will show that $\mathcal{G}$ has the $B_k$ property with respect to the function $h$.

Let $G$ be an $n$-vertex graph in the family $\mathcal{G}$ such that $G$ has at least $\varepsilon n^k$ $k$-cliques. Let $m=\frac{4}{\varepsilon}$. If $n\le m$, then we are done by Nikiforov's result~\cite{2008BLMSNikiforov}. Otherwise, by~\cref{lem:psrl}, there exists $M=M(m,\varepsilon_0)$ such that $V(G)$ admits an $\varepsilon$-regular equitable partition $\mathcal{P}$ with $K+1$ parts $V_{0},V_{1},\ldots,V_{K}$, where $m \le K \le M$. Setting $|V_1|=q$, we can see that $(1-\varepsilon_0)\frac{n}{K}\le q\le \frac{n}{m}$.

We then construct an auxiliary graph $G'$ with vertex set $\{V_{1},V_{2},\ldots,V_{K}\}$, and a pair of vertices $V_{i},V_{j}$ is adjacent in $G'$ if and only if $(V_i,V_j)$ is $\varepsilon_0$-regular and $d(V_i,V_j)>\frac{\varepsilon}{2}$.

\begin{claim}
$G'$ contains a $k$-clique.
\end{claim}

\begin{poc}
Suppose that $G'$ is $K_{k}$-free, then for any $k$-clique in $G$ with vertex set $\{v_1,v_2,\cdots,v_k\}$, at least one of the following four types occurs:

\begin{itemize}
\item At least one of the vertices $v_1,v_2,\ldots,v_k$ belongs to $V_0$. Note that $|V_0|\le \varepsilon_0 n$, the number of $k$-tuples of this type is at most $\varepsilon_0 n^k$.
\item At least two of the vertices $v_1,v_2,\ldots,v_k$ belong to some $V_{i}$, where $i\in [K]$. Note that $|V_i|\le \frac{n}{m}$, the number of the tuples of this type is at most $\frac{n^{k}}{m} \le \frac{\varepsilon}{4} n^k$.
\item There exist some $v_i\in V_a, v_j\in V_b$ such that $V_a$ and $V_b$ are not $\varepsilon_0$-regular, where $i,j\in [k]$ and $a,b\in [K]$. Note that there are at most $\varepsilon_0 K^{2}$ pairs $(V_{i},V_{j})$ being not $\varepsilon_0$-regular, therefore the number of $k$-tuples of this type is at most $\varepsilon_0 n^k$.
\item There exist some $v_i\in V_a, v_j\in V_b$ such that $d(V_a,V_b)<\frac{\varepsilon}{2}$, where $i,j\in [k]$ and $a,b\in [K]$. Then, for fixed $V_a,V_b$, there are at most $\frac{\varepsilon q^{2}}{2}$ choices of $v_i,v_j$, hence the number of $k$-tuples of this type is at most $m^2 \cdot \frac{\varepsilon q^2}{2} \cdot n^{k-2} \le \frac{\varepsilon n^k}{2}$.
\end{itemize}

Hence, the total number of $k$-cliques in $G$ is less than $\varepsilon n^k$, a contradiction.
\end{poc}

Without loss of generality, we can assume that $V_{1},V_{2},\ldots,V_{k}$ form a clique in $G'$. We consider two cases depending on whether some $V_i$ is locally dense.

\medskip

\noindent\textbf{Case 1.} At least one of the clusters $G[V_1],\ldots,G[V_k]$ is $(\varepsilon_1,\varepsilon_2)$-dense. Without loss of generality, we assume that $G[V_1]$ is $(\varepsilon_1,\varepsilon_2)$-dense. We will use induction to show that for $2\le t\le k$, we can find a copy of $K_t[f(\frac{\varepsilon_{2}}{3})^{t-1}q]$ in $V_1$. First, since $G[V_1]$ has density at least $\varepsilon_2$ and $B_2$ property, we can find a $K_2[f(\frac{\varepsilon_{2}}{3})q]$ in $V_1$. Suppose that we have found a copy of $K_t[f(\frac{\varepsilon_{2}}{3})^{t-1}q]$ in $V_1$ with parts $P_1,\ldots,P_t$. Since $G[V_1]$ is $(\varepsilon_1,\varepsilon_2)$-dense and $\eps_1\ll \eps_2$, $G[P_1]$ has density at least $\varepsilon_2$, then we can find a copy of $K_2[f(\frac{\varepsilon_{2}}{3})|P_1|]$ in $P_1$. Thus, this biclique together with $P_2,\ldots,P_t$ contains a copy of $K_{t+1}[f(\frac{\varepsilon_{2}}{3})^{t}q]$ in $V_1$. Taking $t=k$, we get the desired large blowup of $K_k$.

\medskip

\noindent\textbf{Case 2.} We may then assume that none of $G[V_1],\ldots,G[V_k]$ is $(\varepsilon_1,\varepsilon_2)$-dense, then for each $i\in [k]$, we can pick a subset $U_i' \subseteq V_i$ such that $|U_i'|\ge \varepsilon_1q$ and $G[U_{i}']$ has density at most $\varepsilon_2$. Let $q_0=\lceil\varepsilon_1q\rceil$. Pick a subset $U_i$ of $U_i'$  among all possible subgraphs with size $q_0$ which achieve the minimal density, then $G[U_{i}]$ has density at most $\varepsilon_2$.  We will show that there is a copy of $K_k[\delta n]$ such that the $i$th-part is contained in $U_{i}$.

\begin{claim}\label{clm:bic}
For any distinct subsets $A_i\subseteq U_i$, $A_j\subseteq U_j$, with $i,j\in [k]$ and $|A_i|=|A_j|\ge \varepsilon_{3}q_{0}$, we can find subsets $B_i\subseteq A_i$, $B_j\subseteq A_j$ with $|B_i|=|B_j|\ge \varepsilon_{4}|A_{i}|$, such that $(B_i,B_j)$ forms a biclique. 
\end{claim}

\begin{poc}
Since $U_i$ and $U_j$ are $\varepsilon_0$-regular with density at least $\varepsilon/2$ and $|A_i|,|A_j|\ge \varepsilon_3q_0$, we have $d(A_i,A_j)\ge \frac{\varepsilon}{4}$. Hence, $G[A_i\cup A_j]$ has edge density at least $\frac{\varepsilon}{16}$ since $|A_i|=|A_j|$. Thus by the $B_{2}$ property, we can find a copy of $K_2[f(\frac{\varepsilon}{16})|A_i|]$ with parts $X,Y$ in $G[A_i\cup A_j]$.

By symmetry, assume $|X\cap A_i|\ge \frac{|X|}{2}$. Let $X_0=X\cap A_i$. If $|Y\cap A_i|\ge \frac{|Y|}{2}$, letting $Y_0=Y\cap A_i$, we have

\[\varepsilon_2\ge \frac{e(G[U_i])}{\binom{|U_i|}{2}}\ge \frac{e(X_0,Y_0)}{|U_i|^2}\ge \frac{|X||Y|}{4|U_i|^2}\ge \frac{1}{4}\varepsilon_3^2f\left(\frac{\varepsilon}{16}\right)^2,\]
which is a contradiction. Therefore, we have $|Y\cap A_j|\ge \frac{|Y|}{2}$. Let $Y_1=Y\cap A_j$, then $X_0,Y_1$ form a biclique. Finally, just pick $B_i\subseteq X_0$, $B_j\subseteq Y_1$ with $|B_i|=|B_j|\ge \varepsilon_4|A_i|$.
\end{poc}

\begin{claim}\label{clm:sub}
For any subgraph $H$ of $K_k$ on vertex set $[k]$, there are sets $W_{i} \subseteq U_i$ such that

\begin{itemize}
\item $|W_1|=\cdots=|W_k|\ge \varepsilon_4^{e(H)}q_0$; and

\item if $i$ and $j$ are adjacent in $H$, then $(W_i,W_j)$ forms a biclique.

\end{itemize}
\end{claim}

\begin{poc}
We use induction on the number of edges in $H$. When $H$ is a empty graph, taking $W_i=U_i$ suffices.

Assume we have proven the claim for $e(H) \le s$. When $e(H) = s+1$, let $H'$ be the graph obtained by removing one edge $e$ in $H$. By symmetry, we may assume that $e=\{1,2\}$. By the induction hypothesis, we can find sets $W'_{i} \subseteq U_i$ such that

\begin{itemize}
\item $|W'_1|=\cdots=|W'_k|\ge \varepsilon_4^{|H|-1}q_0$; and

\item if $i$ and $j$ are adjacent in $H'$, then $(W'_i,W'_j)$ forms a biclique.
\end{itemize}

By \cref{clm:bic}, we can find a subset $W_1$ of $W_1'$, and subset $W_2$ of $W_2'$ with $|W_1|=|W_2|\ge \varepsilon_4|W_1'|$ such that $(W_1,W_2)$ is a biclique. Finally, for $3\le i\le k$, choose a subset $W_i$ of $W_i'$ with size $|W_1|$. It is obvious that $W_1,\ldots,W_k$ satisfy the condition.
\end{poc}

Applying \cref{clm:sub} with $H=K_k$ shows that we can find a $K_k[\delta n]$ such that the $i$th-part is contained in $U_i$, finishing the proof.
\end{proof}

\subsection{Graphs with bounded VC-dimension} \label{subsec:vc}

Given a set system $\mathcal{F}\subseteq 2^{X},$ a set $S \subseteq X,$ is said to be \emph{shattered} by $\mathcal{F}$ if for every $B \subseteq S,$ there exists a set $F \in \mathcal{F}$ such that $F \cap S = B.$ The VC-dimension of $\mathcal{F}$ is the size of the largest set shattered by $\mathcal{F}.$ Given a graph $G=(V,E)$, for any vertex $v\in V$, let $N_{G}(v)$ be the set of its neighbors. The VC-dimension of $G$ is defined to be the VC-dimension of the set system $\{N_{G}(v)\subseteq V(G):v\in V(G)\}.$

We start with a simple proposition about graphs with VC-dimension $1$.
\begin{prop}\label{Claim:ForbiddenGraphsVC1}
Let $G$ be an $n$-vertex graph with VC-dimension $1$, then the following holds. 
\begin{enumerate}
    \item[\textup{(1)}] If $G$ contains a copy of a triangle with vertices $x_{1},x_{2},x_{3}$, then any vertex $v\in V(G)\setminus\{x_{1},x_{2},x_{3}\}$ is adjacent to at least $2$ vertices of $x_{1},x_{2},x_{3}$.
    \item[\textup{(2)}] If there is a copy of $P_{5}:= y_{1}y_{2}y_{3}y_{4}y_{5}$, then at least one of the pairs $y_{1}y_{4},y_{2}y_{4},y_{2}y_{5}$ is adjacent.
\end{enumerate}
\end{prop}

\begin{proof}
To see (1), if a vertex $v\in V(G)\setminus\{x_{1},x_{2},x_{3}\}$ is not adjacent to two vertices, say $x_1$ and $x_2$, then the set $\{x_1, x_2\}$ is shattered. To see (2), note that if none of the edges $y_{1}y_{4},y_{2}y_{4},y_{2}y_{5}$ are present, then the set $\{y_2, y_4\}$ is shattered.
\end{proof}

We need the following regularity lemma for graphs of bounded VC-dimension.

\begin{theorem}[\cite{2019DCGVC}]\label{thm:VCdimRegularity}
    Let $\gamma \in (0,\frac{1}{4})$ and $G=(V,E)$ be an $n$-vertex graph with VC-dimension $d$. Then $V(G)$ has an equitable partition $V(G)=V_{1}\cup\cdots\cup V_{K}$ with $\frac{8}{\gamma}\le K\le c\big(\frac{1}{\gamma}\big)^{2d+1}$ parts such that all but a $\gamma$-fraction of the pairs of parts are $\gamma$-homogeneous, where $c=c(d)$ is a constant depending only on $d$.    
\end{theorem}

\begin{proof}[Proof of Theorem \ref{thm:VCDim1}]
We first argue that there are graphs with VC-dimension 2 but have no $B_2$ property. Consider a random balanced bipartite $C_4$-free graph with $\Theta(n^{4/3})$ edges. Then its complement has $(\frac{1}{4}-o(1))n^2$ edges, VC dimension at most 2 and with high probability no complete biclique of size $n^{2/3+o(1)}$.

Let $G$ be an $n$-vertex graph with VC-dimension $1$ and $e(G)=c\binom{n}{2}$ with $c>0$. It is left to prove that $G$ contains a $K_2[t]$ with $t=\Omega_c(n)$. We consider two cases.

\medskip

\noindent\textbf{Case 1. $c \ge \frac{99}{100}.$} 

\begin{claim}\label{claim:Dense}
    $G$ contains a subgraph $G'$ with $|G'|\ge \frac{9n}{10}$ and $\delta(G')\ge\frac{2n}{3}$.    
\end{claim}
\begin{poc}
    We iteratively define a sequence of graphs $G_{0}=G\supseteq G_{1}\supseteq G_{2}\supseteq\cdots$, where for each $i\ge 1$, we obtain $G_{i}$ by removing one vertex with degree less than $\frac{2n}{3}$ from $G_{i-1}$ if it exists, otherwise we stop. Note that the number of non-adjacent pair of vertices in $G$ is at most $\frac{1}{100} \binom{n}{2}$ and each removal decreases the number of non-edges by at least $\frac{n}{3}.$ Hence, the process will stop after fewer than $\frac{3}{100} n$ steps. Let $G'$ be the resulting subgraph. It follows that $\delta(G')\ge\frac{2n}{3}$ and $|V(G')|\ge \frac{9}{10} n$.  
\end{poc}

The following claim implies that the graph $G'$ is a complete multipartite graph.

 \begin{claim}\label{claim:sameNeighbor}
   For any pair of non-adjacent vertices $u,v\in V(G')$, $N_{G'}(u)=N_{G'}(v)$.
 \end{claim}
 \begin{poc}
     Let $u,v\in V(G')$ be a pair of vertices such that $uv\notin E(G')$. Suppose that there exists a vertex $w$ such that $uw\in E(G')$ and $vw\notin E(G')$, note that $u$ and $w$ have a common neighbor $z\in V(G')$ as $\delta(G')\ge\frac{2n}{3}$. That means $v$ can be adjacent to at most one vertex of the triangle $uwz$, a contradiction to~\cref{Claim:ForbiddenGraphsVC1}.
 \end{poc}
Since $G'$ is a complete multipartite graph and $\delta(G') \ge \frac{2n}{3},$ it follows that each of the parts has size at most $\frac{n}{3}.$ Since $|V(G')| \ge \frac{9n}{10},$ the parts of $G'$ can be split into two sides such that each of the sides has size at least $\frac{n}{5},$ implying that $G$ contains $K_2[\frac{n}{5}]$ as desired.

\medskip

\noindent\textbf{Case 2. $c < \frac{99}{100}$.}
 
Let $\gamma=10^{-5}$. By~\cref{thm:VCdimRegularity}, we can partition $V(G)$ into $V_{1}\cup V_{2}\cup\cdots\cup V_{K}$ with $K\le c'\big(\frac{1}{\gamma}\big)^{3}$ for some $c'>0$, such that all but a $\gamma$-fraction of the pairs of parts are $\gamma$-homogeneous. Since $c \binom{n}{2} \le e(G) \le \frac{99}{100}\binom{n}{2},$ it is easy to see that there are three parts $V_{i},V_{j},V_{k}$ such that $(V_{i},V_{j})$ is $\gamma$-dense, $(V_{j},V_{k})$ is $\gamma$-sparse and $|V_i| = |V_j| = |V_k| = m \ge \frac{n}{2K}$. Similarly as in~\cref{claim:Dense}, we remove any vertex from $V_{j}$ with at most $\frac{4m}{5}$ neigbours in $V_i$ or at least $\frac{m}{5}$ neighbors in $V_k.$ Let $U_j \subseteq V_j$ be the remaining set and observe that $|U_j| \ge (1 - 10\gamma) m$ and for every $u \in U_j,$ we have $d_{V_i}(u) \ge \frac{4m}{5}$ and $d_{V_k}(u) \le \frac{m}{5}.$

Note that there are at least $\frac{4m}{5} |U_j| \ge \frac{3m^2}{4}$ edges between $U_j$ and $V_i$ so we can pick a vertex $v \in V_i$ with $|N_{U_{j}}(v)|\ge \frac{3m}{4}$. We first claim that $N_{U_{j}}(v)$ is an independent set. Indeed, if there is an edge $u_{1}u_{2}$ in $N_{U_j}(v)$, then every vertex in $V_{k}$ should be adjacent to at least one of $u_{1},u_{2}$ by~\cref{Claim:ForbiddenGraphsVC1}~(1). However, then one of $u_1, u_2$ has at least $\frac{m}{2}$ neighbors in $V_k,$ a contradiction.

Moreover, we have the following claim.
\begin{claim}
    For any distinct vertices $u_{1},u_{2}\in N_{U_{j}}(v)$, either $N_{V_{i}}(u_{1})\subseteq N_{V_{i}}(u_{2})$, or $N_{V_{i}}(u_{2})\subseteq N_{V_{i}}(u_{1})$. 
\end{claim}
\begin{poc}
Suppose this is not the case so there is a vertex $v_{1}\in N_{V_{i}}(u_{1})\setminus N_{V_{i}}(u_{2})$ and a vertex $v_{2}\in N_{V_{i}}(u_{2})\setminus N_{V_{i}}(u_{1})$. Then $v_{1}u_{1}vu_{2}v_{2}$ forms a copy of $P_{5}$. However, none of the edges $v_{1}u_{2},u_{1}u_{2},u_{1}v_{2}$ are present, contradicting~\cref{Claim:ForbiddenGraphsVC1}~(2).
\end{poc}
Now, consider a vertex $u \in N_{U_{j}}(v)$ minimizing $|N_{V_i}(u)|$. Then, for any $u' \in N_{U_{j}}(v)$, we have $N_{V_i}(u') \supseteq N_{V_i}(u)$ implying the existence of $K_2[\frac{3m}{4}]$ in $G$.

This completes the proof of Theorem \ref{thm:VCDim1}.
\end{proof}

\section{Concluding remarks}\label{sec:conclusion}
Recall that $g(k)$ is the minimum $r\in \I N$ such that if any $n$-vertex ordered graph $G$ with no induced monotone $P_{2k}$ satisfies $\varrho_r(\overline{G})>0$, then $\overline{G}$ contains a copy of $K_2[\Omega(\frac{n}{\log n})]$. We determined $g(k)$ up to a factor of $2$. It would be interesting to close this gap. To this end, we make the following conjecture on the Ramsey problem for $f(k)$ defined in~\cref{def:dependency-digraph}.

\begin{conjecture}\label{fk}
$f(k)=\floor{\frac{k^{2}}{4}} +1$.
\end{conjecture}

By our construction in~\cref{section:ConstructionBEGraph}, the conjecture above, if true, would be optimal, implying that $g(k)=f(k)=\floor{\frac{k^{2}}{4}} +1$. With the assistance of computers, we checked that \cref{fk} is true for $k\le 6$.


An important theme we systematically investigate in this paper is that of ``replacing edge density by clique density''. We believe that this theme merits further study. A recent result of Holmsen~\cite{2020Holmsen} is of this flavor. It is known that for any intersection graph $G$ of axis-aligned boxes in $\mathbb{R}^d$, if $G$ has edge density larger than $1-\frac{1}{d}$, then it contains a linear-size clique. A corollary of Holmsen's result shows that it suffices to require positive $K_{d+1}$-density.

Tomon~\cite{2022DCGSharpString} proved that if $G$ is the complement of a string graph and has edge density larger than $\frac{3}{4}$, then $G$ contains a linear-size biclique. Following our theme here, we propose the following conjecture. 

\begin{conjecture}
Let $\eps>0$ and $G$ be the complement of a string graph. If $\varrho_{4s+1}(G)>0$, then $G$ contains a $K_{s+1}[\delta n]$, where $\delta=\delta(\eps)>0$.
\end{conjecture}

We remark that taking disjoint union of tight examples in~\cite{2022DCGSharpString} shows that if the conjecture is true, it is optimal in the sense that positive $K_{4s}$-density does not suffices to guarantee linear-size blowup of $K_{s+1}$.

In~\cref{thm:VCDim1}, we show that linear-size biclique is guaranteed only in dense graphs with VC dimension 1. How large a biclique do dense graphs with larger VC-dimension contain? Anthony, Brightwell and Cooper~\cite{1995DMThresholdBVC} proved that the random graph $G(n, p)$ with edge probability $p = 1 - n^{-\frac{1}{d}+O(\frac{1}{d^{2}})}$, with high probability, has VC-dimension at most $d$ and is $K_2[t]$-free, where $t=n^{\frac{1}{d}+O(\frac{1}{d^2})}$. Hence, there exist graphs with $(1 - o(1)) \binom{n}{2}$ edges, VC-dimension at most $d$ and no biclique of size $n^{\frac{1}{d}+O(\frac{1}{d^2})}$. On the other hand, we can show that any dense graph with VC-dimension $d$ contains a biclique of size $\Omega(n^{\frac{1}{d+1}})$.

\begin{prop}\label{thm:GeneralBVC}
    Let $d\ge 2$ be positive integers and $G$ be an $n$-vertex graph with VC-dimension $d$. If $e(G)\ge c\binom{n}{2}$ for some $c>0$, then $G$ contains a copy of $K_{2}[c'n^{\frac{1}{d+1}}]$ for some $c'$.
\end{prop}

A system $\C X\subseteq 2^V$ is \emph{$s$-separated} if for any $F,F'\in\C X$, $|F\triangle F'|\ge s$.

\begin{lemma}[Haussler~\cite{1995PackingLemma}]\label{lemma:Packing}
     Let $\mathcal{F}\subseteq 2^{V}$ be a set system such that $\mathcal{F}$ has VC-dimension $d$. If $\mathcal{X}\subseteq \mathcal{F}$ is $s$-separated, then $|\mathcal{X}|\le c_{1}(\frac{|V|}{s})^{d}$, where $c_{1}=c_{1}(d)$.
\end{lemma}

\begin{proof}[Proof of~\cref{thm:GeneralBVC}]
Let $G$ be an $n$-vertex graph with VC-dimension $d$ and $e(G)\ge c\binom{n}{2}$. Deleting low degree vertices, we get from $G$ a subgraph $G'$ with $\delta(G')\ge cn/2$. Let $\mathcal{F}':=\{N_{G'}(v)\subseteq V(G'): v\in V(G')\}$, the VC-dimension of $\mathcal{F}'$ is also at most $d$.

Set $s=|G'|^{1-\frac{1}{d+1}}$ and take a maximal $s$-separated subfamily $\mathcal{X}=\{{N_{G'}(v_{1}),N_{G'}(v_{2}),\ldots,N_{G'}(v_{m})}\}\subseteq \C F'$. We have $m\le c_{1}(\frac{|G'|}{s})^{d}$ by~\cref{lemma:Packing}. Then, partition $V(G')$ into $V_{1}\sqcup V_{2}\sqcup\cdots\sqcup V_{m}$ with $v_{i}\in V_{i}$ as follows. We put $N_{G'}(v)\in\mathcal{F'}\setminus \mathcal{X}$ into $V_{i}$ if $i$ is the smallest index such that $|N_{G'}(v)\triangle N_{G'}(v_{i})|\le s$. Note that after partitioning, each set in $\mathcal{F'}\setminus \mathcal{X}$ belongs to some part $V_{i}$. Moreover, every pair of vertices $u,v$ with their neighborhoods in the same part satisfies $|N_{G'}(u)\triangle N_{G'}(v)|\le 2s$ by the triangle inequality.

Now pick arbitrary vertices $v_{1},v_{2},\ldots,v_{q}$ from some part $V_{i}$ with $|V_{i}|=\Omega(\frac{n}{m})$, where $q=\Omega(n^{\frac{1}{d+1}})$. As $\delta(G')> (2q+1)s+q$, we see that $\big|\bigcap_{h\in [q]}N_{G'}(v_{h})\big|\ge |N_{G'}(v_{1})|-(2q+1)s\ge q$.
Note that $|V_{i}|\ge \Omega(n^{1-\frac{1}{d+1}})>q$, thus there is a biclique of size $q$.
\end{proof}

\bibliographystyle{abbrv}
\bibliography{2.bib}

\end{document}